\documentclass[a4paper]{article}
\usepackage{makecell}

\usepackage[numbers,sort&compress]{natbib}

\makeatletter
\let\c@figure\c@table
\let\ftype@figure\ftype@table

\usepackage{enumerate}
\usepackage{calc}

\usepackage{lmodern} 

\usepackage{graphicx} 

\usepackage{amsmath, accents}
\usepackage{amssymb,tikz}
\usepackage{pict2e}
\usepackage{amsthm}
\usepackage{amscd}
\usepackage{mathrsfs}
\usepackage{todonotes}
\usetikzlibrary{calc} 

\usepackage{yfonts}

\usepackage{marvosym}
\usepackage[percent]{overpic}

\newtheorem{theorem}{Theorem} [section]

\newtheorem{lemma}[theorem]{Lemma}

\theoremstyle{definition}

\newtheorem{remark}{Remark}

\newcommand{\C}{\mathbb{C}}

\newcommand{\R}{\mathbb{R}}
\newcommand{\N}{\mathbb{N}}

\let\oldbibliography\thebibliography
\renewcommand{\thebibliography}[1]{\oldbibliography{#1}
\setlength{\itemsep}{-0.5pt}}

\def\XXint#1#2#3{{\setbox0=\hbox{$#1{#2#3}{\int}$}
\vcenter{\hbox{$#2#3$}}\kern-.5\wd0}}

\usepackage{tikz}
\usetikzlibrary{arrows}
\usetikzlibrary{decorations.pathmorphing}
\usetikzlibrary{decorations.markings}
\usetikzlibrary{patterns}
\usetikzlibrary{automata}
\usetikzlibrary{positioning}
\usepackage{tikz-cd}
\tikzset{->-/.style={decoration={
				markings,
				mark=at position #1 with {\arrow{latex}}},postaction={decorate}}}
	
	\tikzset{-<-/.style={decoration={
				markings,
				mark=at position #1 with {\arrowreversed{latex}}},postaction={decorate}}}
				
\tikzset{
	master/.style={
		execute at end picture={
			\coordinate (lower right) at (current bounding box.south east);
			\coordinate (upper left) at (current bounding box.north west);
		}
	},
	slave/.style={
		execute at end picture={
			\pgfresetboundingbox
			\path (upper left) rectangle (lower right);
		}
	}
}

\usetikzlibrary{shapes.misc}\tikzset{cross/.style={cross out, draw, 
         minimum size=2*(#1-\pgflinewidth), 
         inner sep=0pt, outer sep=0pt}}

\usepackage{pgfplots}
	
\allowdisplaybreaks	

\numberwithin{equation}{section}

\def\bigO{{\cal O}}

\makeatletter
\newcommand{\oset}[3][0ex]{%
  \mathrel{\mathop{#3}\limits^{
    \vbox to#1{\kern-2\ex@
    \hbox{$\scriptstyle#2$}\vss}}}}
\makeatother

\usepackage{tocloft}
\usepackage[colorlinks=true]{hyperref}
\hypersetup{urlcolor=blue, citecolor=red, linkcolor=blue}

\begin{document}
\title{Balayage of measures: behavior near a corner}
\author{Christophe Charlier\footnote{Centre for Mathematical Sciences, Lund University, 22100 Lund, Sweden. E-mail: christophe.charlier@math.lu.se} \, and Jonatan Lenells\footnote{Department of Mathematics, KTH Royal Institute of Technology, 10044 Stockholm, Sweden. E-mail: jlenells@kth.se}}
\date{}
\maketitle

\begin{abstract}
We consider the balayage of a measure $\mu$ defined on a domain $\Omega$ onto its boundary $\partial \Omega$. Assuming that $\Omega$ has a corner of opening $\pi \alpha$ at a point $z_0 \in \partial \Omega$ for some $0 < \alpha \leq 2$ and that $d\mu(z) \asymp |z-z_{0}|^{2b-2}d^{2}z$ as $z\to z_0$ for some $b > 0$, we obtain the precise rate of vanishing of the balayage of $\mu$ near $z_{0}$. The rate of vanishing is universal in the sense that it only depends on $\alpha$ and $b$. 
We also treat the case when the domain has multiple corners at the same point. Moreover, when $2b\leq \frac{1}{\alpha}$, we provide explicit constants for the upper and lower bounds.
\end{abstract}
\noindent
{\small{\sc AMS Subject Classification (2020)}: 31A15, 31A20, 31A05}

\noindent
{\small{\sc Keywords}: Balayage measure, harmonic measure, boundary behavior.}


%
%

\section{Introduction}

Balayage measures were introduced by Henri Poincar\'e in the late 19th century as a tool to solve the Laplace equation \cite{P1899}. Given a measure $\mu$ on a domain $\Omega$, the balayage (sweeping) of $\mu$ onto $\partial \Omega$ is a measure $\nu$ on $\partial \Omega$ whose potential outside $\Omega$ coincides with that of $\mu$ (up to a constant if $\Omega$ is unbounded).
More precisely, given a bounded Jordan domain $\Omega$ and a non-negative measure $\mu$ with compact support in $\Omega$, the balayage measure $\nu := \mathrm{Bal}(\mu,\partial \Omega)$ is defined as the unique measure supported on $\partial \Omega$ such that $\nu(\partial \Omega)=\mu(\Omega)$, $\nu(P)=0$ for every Borel set $P$ of zero capacity, and such that 
\begin{align*}
\int_{\partial \Omega} \log \frac{1}{|z-w|} \, d\nu(w) = \int_{\Omega} \log \frac{1}{|z-w|} \, d\mu(w)
\end{align*}
holds for quasi-every $z \in \C\setminus \Omega$ (see e.g. \cite[Theorem II.4.7]{SaTo}). 
 Balayage measures can also be defined in terms of the harmonic measure (see \eqref{nuharmonicmeasure} below). The concept of balayage plays a significant role in potential theory, see e.g. \cite{D2001, Land1972}. Various applications and generalizations of balayage theory can be found in \cite{AR2017, AS2019, BBMP2009, C2023, GK2021, G1997, G2004, GR2018, GS1994, NW2023, Z2022, Z2023, Z2023b}.

\medskip Let $B_r(z)$ denote the open disk of radius $r$ centered at $z$.
In this paper we deal with the following question, which is relevant for example in the study of two-dimensional Coulomb gases as explained in Section \ref{section:application}. 

\begin{enumerate}
\item[] \textbf{Question:} If 
\begin{itemize}
\item $\Omega$ has a H\"older-$C^1$ corner of opening $\pi \alpha$ at a point $z_0 \in \partial \Omega$ for some $0 < \alpha \leq 2$, and
\item $d\mu(z) \asymp |z-z_{0}|^{2b-2}d^{2}z$ as $z\to z_0$ ($z \in \Omega$) for some $b > 0$, where $d^{2}z$ is the two-dimensional Lebesgue measure,
\end{itemize}
what is the behavior of $\nu(\partial \Omega \cap B_r(z_0))$ as $r\to 0$?  
\end{enumerate}

It has been conjectured in \cite{C2023} that 
\begin{align}\label{nuestimate intro}
\nu(\partial \Omega \cap B_r(z_0)) \asymp \begin{cases}
r^{\min\{2b,\frac{1}{\alpha}\}}, & \mbox{if } 2b \neq \frac{1}{\alpha}, \\
r^{2b}\log \frac{1}{r}, & \mbox{if } 2b = \frac{1}{\alpha},
\end{cases} \qquad \mbox{as } r \to 0.
\end{align}
Here we prove this conjecture. In particular, we establish that the rate of vanishing of $\nu(\partial \Omega \cap B_r(z_0))$ as $r\to 0$ only depends on $\alpha$ and $b$, but is otherwise universal, in the sense that it is independent of other properties of $\Omega$ and $\mu$. In fact, we will strengthen (\ref{nuestimate intro}) in several ways:

\begin{enumerate}
\item We show that (\ref{nuestimate intro}) holds for more general open sets $\Omega$ than Jordan domains. In particular, $\Omega$ does not have to be connected, and the connected components of $\Omega$ do not have to be simply connected. 

\item We obtain bounds on the implicit constants in (\ref{nuestimate intro}) whenever $2b \leq 1/\alpha$. For example, if $2b < 1/\alpha$ and $d\mu(z) = (1 + o(1)) |z-z_{0}|^{2b-2}d^{2}z$ as $z\to z_0$, then we show that, for any $\epsilon >0$,
$$(1-\epsilon) \frac{\tan(\pi \alpha b)}{2b^2} r^{2b} \leq \nu(\partial \Omega \cap B_r(z_0)) \leq (1+\epsilon) \frac{\pi \alpha}{2b} \bigg(1  + \frac{16 b}{\pi (\frac{1}{\alpha} - 2b)}\bigg)  r^{2b}$$
for all small enough $r>0$. Since $\frac{\tan(\pi \alpha b)}{2b^2} = \frac{\pi \alpha}{2b} + O(b)$ as $b \to 0$, this formula shows that $\nu(\partial \Omega \cap B_r(z_0))$ behaves like $\frac{\pi \alpha}{2b} r^{2b}$ in the limit of small $b$, i.e., not only is the order of vanishing of the balayage measure $\nu$ near the corner a universal quantity depending only on $\alpha$ and $b$, but the constant prefactor is also universal in the small $b$ limit. 

\item We show that (\ref{nuestimate intro}) remains true also in the case when $\Omega$ has multiple corners at a single point $z_0$, provided that $\pi \alpha$ is the largest of the opening angles. More precisely, if $\Omega$ is an open set such that $\partial \Omega \cap B_r(z_0)$ for small $r>0$ is a disjoint union of $m$ domains with H\"older-$C^1$ corners at $z_0$ with opening angles $\pi \alpha_j \in (0,2\pi]$, $j=1,\dots, m$, and $\Omega \setminus B_r(z_0)$ is a disjoint union of finitely connected Jordan domains, then (\ref{nuestimate intro}) holds with $\alpha := \max_j \alpha_j$.
\end{enumerate}

The situation where $\Omega$ has an outward-pointing cusp (i.e. the case $\alpha=0$) deserves a separate analysis and is treated in the companion paper \cite{CLcusp}.

\section{Main results}

Let $\C^* = \C \cup \{\infty\}$ be the Riemann sphere and let $\mathbb{D} = \{z \in \C : |z| < 1\}$ be the open unit disk. 
A {\it Jordan curve in $\C^*$} is the image of the unit circle under an injective continuous function $\partial \mathbb{D} \to \C^*$; in other words, it is a non-self-intersecting loop in $\C^*$.
If $\Omega$ is a simply connected open subset of $\C^*$, then $\Omega$ is a {\it Jordan domain} if $\partial \Omega$ is a Jordan curve in $\C^*$. 
If $\Omega$ is an open connected subset of $\C^*$ such that $\partial \Omega$ is a finite union of pairwise disjoint Jordan curves, then we say that $\Omega$ is a {\it finitely connected Jordan domain} in $\C^*$.

A {\it Jordan arc in $\C^*$} is the image of a closed interval $I \subset \R$ under an injective continuous function $I \to \C^*$. A Jordan arc $C \subset \C$ is of class $C^{1,\gamma}$, $0 < \gamma \leq 1$, if it has a parametrization $C: w(t), 0 \leq t \leq 1$ such that the derivative $w'(t)$ exists, is nonzero, and is H\"older continuous with exponent $\gamma$ on $[0,1]$.


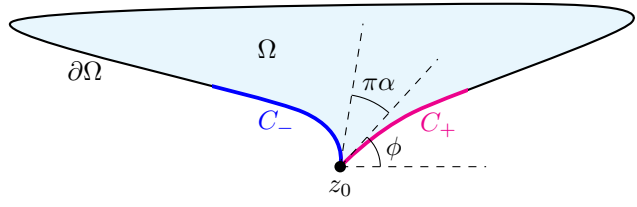
\begin{figure}
\begin{center}
\begin{tikzpicture}[scale = 2.4]
  \coordinate (origin) at (0, 0);

  \coordinate (z0) at (0,0);
  \coordinate (c2) at (0.4, 0.3);
  \coordinate (c3) at (1.6, 0.8);
  \coordinate (c4) at (0.9, 0.9);
  \coordinate (c5) at (-1.8, 0.8);
  \coordinate (c6) at (-0.2, 0.3);
  \coordinate (c7) at (0.4, -0.45);

  \draw[black, thick, fill=cyan!8] plot [smooth] coordinates {(z0) (c2) (c3) (c4) (c5) (c6) (z0)};

  \begin{scope}
    \clip (0.01,0) rectangle (0.7,0.7);
  \draw[magenta, thick, line width = 0.5mm] plot [smooth] coordinates {(z0) (c2) (c3) (c4) (c5) (c6) (z0)};
  \end{scope}
  \node[magenta] at (0.54,0.23){$C_+$};

  \begin{scope}
    \clip (0.015,0) rectangle (-0.7,0.7);
  \draw[blue, thick, line width = 0.5mm] plot [smooth] coordinates {(z0) (c2) (c3) (c4) (c5) (c6) (z0)};
  \end{scope}
  \node[blue] at (-0.35,0.24){$C_-$};

  \draw[black, dashed] (z0) -- (0:0.8);
  \draw[black, dashed] (z0) -- (48:0.8);
  \draw[black, dashed] (z0) -- (81:0.8);

\draw[black] (origin)+(0:0.22) arc (0:48:0.22);
  \node at (19:0.31){$\phi$};

\draw[black] (origin)+(48:0.4) arc (48:81:0.4);
  \node at (65:0.5){$\pi \alpha$};

  \node at (-0.4,0.65){$\Omega$};

  \node at (-1.4,0.53){$\partial\Omega$};

\fill (origin) circle (0.9pt);
\node at (0,-0.12){$z_0$};

\end{tikzpicture} 
 \caption{\label{cornerfigure} Illustration of a corner of opening $\pi \alpha$ at $z_0$.}
 \end{center}
\end{figure}

Let $\partial \Omega$ be the boundary of a finitely connected Jordan domain $\Omega$ in $\C^*$. 
We say that $\Omega$ has a {\it corner} of opening $\pi \alpha$, $0 < \alpha \leq 2$, at $z_0 \in \partial \Omega \cap \C$ if there are closed Jordan arcs $C_\pm \subset \partial \Omega$ ending at $z_0$ and lying on different sides of $z_0$ such that
\begin{align}\label{arg def angle}
\arg(z - z_0) \to \begin{cases} 
\phi & \text{as $C_+ \ni z \to z_0$}, \\
\phi + \pi \alpha & \text{as $C_- \ni z \to z_0$}, 
\end{cases}
\end{align}
for some $\phi$, and such that $\{z_{0}+\rho e^{i(\phi+\frac{\pi \alpha}{2})}:\rho\in (0,\rho_{0})\}\subset \Omega$ for some small enough $\rho_{0}>0$, see Figure \ref{cornerfigure}. In \eqref{arg def angle}, the branch is chosen so that $\arg(\cdot-z_{0})$ is continuous in $(\bar{\Omega} \cap B_{\rho_{0}}(z_{0}))\setminus \{z_{0}\}$. 
The situation where $\Omega$ has an inward-pointing cusp is covered by the case $\alpha=2$. 
We say that the corner is {\it H\"older-$C^1$} if the Jordan arcs $C_\pm$ can be chosen to be of class $C^{1,\gamma}$ for some $\gamma>0$.

If $\Omega \subset \C^*$ is an open set and $\mu$ is a non-negative measure on $\Omega$, then the balayage $\nu := \mathrm{Bal}(\mu,\partial \Omega)$ of $\mu$ onto $\partial \Omega$ is the measure on $\partial \Omega$ defined whenever it exists by
\begin{align}\label{nuharmonicmeasure}
\nu(E) = \int_\Omega \omega(z, E, \Omega) d\mu(z) \qquad \text{for Borel subsets $E$ of $\partial \Omega$},
\end{align}
where $\omega(z, E, \Omega)$ is the harmonic measure of $E$ at $z$ in $\Omega$. We assume that $\omega(\cdot;E,\Omega)$ is $\mu$-measurable, so that the integral \eqref{nuharmonicmeasure} is well-defined. (We recall that $\omega(z, E, \Omega)$ can be interpreted as the probability that a Brownian motion starting at $z$ exits $\Omega$ at a point in $E$, see e.g. \cite[p. 73]{GM2005}.)

Let $d^2z = dxdy$ denote the Lebesgue measure on $\C$.
If $\mu$ is a non-negative measure on $\Omega$ and $b > 0$, then we write 
\begin{align}\label{muc0o1}
d\mu(z) = (1 +o(1)) |z-z_{0}|^{2b-2}d^{2}z \qquad \text{as $z\to z_0 \in \partial \Omega$}
\end{align}
if for every $\epsilon > 0$, there exists a radius $\rho_0 >0$ such that
\begin{align}\label{muc0o1eps}
\bigg|\mu(A) - \int_A |z-z_{0}|^{2b-2}d^{2}z\bigg| \leq \epsilon \int_A |z-z_{0}|^{2b-2}d^{2}z
\end{align}
for all measurable subsets $A$ of $\Omega \cap B_{\rho_0}(z_0)$.
Similarly, we write
$$d\mu(z) \asymp |z-z_{0}|^{2b-2}d^{2}z \qquad \text{as $z\to z_0 \in \partial \Omega$}$$
if there exists a $\rho_0 >0$ and constants $c_1, c_2 > 0$ such that
\begin{align}\label{c1muc2}
c_1 \int_A |z-z_{0}|^{2b-2}d^{2}z \leq \mu(A) \leq c_2 \int_A |z-z_{0}|^{2b -2}d^{2}z
\end{align}
for all measurable subsets $A$ of $\Omega \cap B_{\rho_0}(z_0)$.
If $f(r)$ and $g(r)$ are two positive functions defined for all small enough $r > 0$, then we write 
$$f(r) \asymp g(r) \qquad \text{as $r \to 0$}$$
if there exists an $r_0 > 0$ and constants $c_1, c_2 > 0$ such that 
$$c_1 g(r) \leq f(r) \leq c_2 g(r) \qquad \text{for all $r \in (0, r_0)$}.$$

\subsection{A single corner}
We first state our main result in the case of a single corner at $z_0$. The result will then be generalized in Section \ref{subsection:main results for several corners} to the case of an arbitrary finite number of corners at $z_0$. In the case of a single corner, our main result is the following. 
We use $C>0$ and $c>0$ to denote generic strictly positive constants. 

\begin{theorem}[A single corner]\label{mainth}
Let $\Omega$ be a finitely connected Jordan domain in $\C^*$. Let $0 < \alpha \leq 2$ and suppose $\Omega$ has a H\"older-$C^1$ corner of opening $\pi \alpha$ at a point $z_0 \in \partial \Omega \cap \C$.  
Let $\mu$ be a non-negative measure of finite total mass on $\Omega$ such that $d\mu(z) = (1 +o(1)) |z-z_{0}|^{2b-2}d^{2}z$ as $z\to z_0$ for some $b > 0$. 
For every $\epsilon > 0$, the balayage $\nu := \mathrm{Bal}(\mu,\partial \Omega)$ of $\mu$ onto $\partial \Omega$ obeys the following inequalities for all sufficiently small $r > 0$:
\begin{align}\nonumber
 (1-\epsilon) \frac{\tan(\pi \alpha b)}{2b^2} r^{2b} & \leq \nu(\partial \Omega \cap B_r(z_0)) \leq (1+\epsilon) \frac{\pi \alpha}{2b} \bigg(1  + \frac{16 b}{\pi (\frac{1}{\alpha} - 2b)}\bigg)  r^{2b} && \text{if $2b < \frac{1}{\alpha}$}, 
	\\ \nonumber 
 (1-\epsilon) \frac{2}{\pi b} r^{2b} \log(\tfrac{1}{r})
& \leq
 \nu(\partial \Omega \cap B_r(z_0))
\leq (1+\epsilon) \frac{4}{b} r^{2b} \log(\tfrac{1}{r}) &&  \text{if $2b= \frac{1}{\alpha}$}, 	\\ \label{nubounds}
c r^{\frac{1}{\alpha}} & \leq
 \nu(\partial \Omega \cap B_r(z_0))
\leq C r^{\frac{1}{\alpha}} & &  \text{if $2b > \frac{1}{\alpha}$}.
\end{align}
In particular, if $d\mu(z) \asymp |z-z_{0}|^{2b-2}d^{2}z$ as $z\to z_0$, then
\begin{align}\label{nuestimate}
\nu(\partial \Omega \cap B_r(z_0)) \asymp \begin{cases}
r^{2b} & \text{if $2b < \frac{1}{\alpha}$}, \\
r^{2b}\log \frac{1}{r} &  \text{if $2b= \frac{1}{\alpha}$}, \\
r^{\frac{1}{\alpha}} &  \text{if $2b > \frac{1}{\alpha}$},
\end{cases}
\qquad \text{as $r \to  0$.}
\end{align}

\end{theorem}

The proof of Theorem \ref{mainth} is presented in Section \ref{proofsec}.

\begin{remark}
Theorem \ref{mainth} implies in particular that  the density of $\nu$ never blows up or vanishes at ``standard" points where $\partial \Omega$ admits a tangent ($\alpha=1$) and where $\frac{d\mu(z)}{d^{2}z} \asymp 1$ ($b=1$). 
\end{remark}

\begin{remark}
In \cite[Theorem 2.16 (i) and Remark 2.18]{C2023}, the balayage measure $\nu$ was computed exactly in the case when $\Omega = \{re^{i\theta}:0<r<a, 0<\theta<\pi \alpha\}$ is a circular sector of radius $a > 0$ and opening angle $\pi \alpha \in (0,2\pi]$, and $\mu$ is the measure $d\mu(z) = |z|^{2b-2}d^{2}z$ with $b>0$. We will use this fact to obtain the lower bounds on $\nu(\partial \Omega \cap B_r(z_0))$ in (\ref{nubounds}). 
\end{remark}

\begin{remark}
Let $c_{1},c_{2}$ be the largest possible constants and $C_{1},C_{2}$ the smallest possible constants such that, for every $\epsilon >0$ and $\nu$ as in Theorem \ref{mainth}, there exists $r_0 >0$ such that for all $r \in (0, r_0)$, the inequalities
\begin{align}\nonumber
 (1-\epsilon) c_{1} r^{2b} & \leq \nu(\partial \Omega \cap B_r(z_0)) \leq (1+\epsilon) C_{1}  r^{2b} && \text{if $2b < \frac{1}{\alpha}$}, 
	\\ \nonumber 
 (1-\epsilon) c_{2} r^{2b} \log(\tfrac{1}{r})
& \leq
 \nu(\partial \Omega \cap B_r(z_0))
\leq (1+\epsilon) C_{2} r^{2b} \log(\tfrac{1}{r}) &&  \text{if $2b= \frac{1}{\alpha}$},
\end{align}
hold. The inequalities \eqref{nubounds} imply that
\begin{align*}
\frac{\tan(\pi \alpha b)}{2b^2} \leq c_{1}, \qquad C_{1} \leq \frac{\pi \alpha}{2b} \bigg(1  + \frac{16 b}{\pi (\frac{1}{\alpha} - 2b)}\bigg), \qquad  \frac{2}{\pi b} \leq c_{2}, \qquad C_{2} \leq \frac{4}{b}.
\end{align*}
On the other hand, it follows from \cite[Remark 2.18]{C2023} that 
\begin{align*}
c_{1} \leq \frac{\tan(\pi \alpha b)}{2b^2}, \qquad c_{2} \leq \frac{2}{\pi b}.
\end{align*}
This implies that the constants $\frac{\tan(\pi \alpha b)}{2b^2}$ and $\frac{2}{\pi b}$ appearing on the left-hand sides in \eqref{nubounds} are the best possible. Finding $C_{1}$ and $C_{2}$ is an interesting problem for future research.
\end{remark}

\subsection{Multiple corners}\label{subsection:main results for several corners}
The following theorem treats the case when $\Omega$ has multiple corners at $z_0$, see Figure \ref{figure:multicorner}.

\begin{theorem}[Multiple corners at the same point]\label{mainth2}
Let $\Omega$ be an open subset of $\C^*$, let $z_0 \in \partial \Omega \cap \C$, and let $m \geq 1$ be an integer. Suppose there is a radius $\rho_0 > 0$ such that

\begin{enumerate}[$(i)$]
\item the open set $\Omega \cap B_{\rho_0}(z_0)$ has $m$ connected components $\{U_j\}_1^m$ such that $\bar{U}_j \cap \bar{U}_k = \{z_0\}$ whenever $j \neq k$ and each component $U_j$ has a H\"older-$C^1$ corner at $z_0$ of opening angle $\pi \alpha_j \in (0,2\pi]$, and

\item $\Omega \setminus \overline{B_{\rho_0}(z_0)}$ is a disjoint union of finitely connected Jordan domains.
\end{enumerate}
Let $\mu$ be a non-negative measure of finite total mass on $\Omega$ such that $d\mu(z) = (1 +o(1)) |z-z_{0}|^{2b-2}d^{2}z$ as $z\to z_0$ for some $b > 0$. 
Let
$$\alpha := \displaystyle{\max_{1\leq j \leq m}} \alpha_j.$$

For every $\epsilon > 0$,  the balayage $\nu := \mathrm{Bal}(\mu,\partial \Omega)$ of $\mu$ onto $\partial \Omega$ obeys the following  inequalities for all sufficiently small $r > 0$:
\begin{align}\nonumber
(1-\epsilon) \sum_{j=1}^m \frac{\tan(\pi \alpha_j b)}{2b^2} r^{2b} & \leq \nu(\partial \Omega \cap B_r(z_0)) \leq (1+\epsilon) \sum_{j=1}^m \frac{\pi \alpha_j}{2b} \bigg(1  + \frac{16 b}{\pi (\frac{1}{\alpha_j} - 2b)}\bigg)  r^{2b} & & \text{if $2b < \frac{1}{\alpha}$}, 
	\\ \nonumber 
m_{\alpha} (1-\epsilon) \frac{2}{\pi b} r^{2b} \log(\tfrac{1}{r})
& \leq
 \nu(\partial \Omega \cap B_r(z_0))
\leq m_{\alpha} (1+\epsilon) \frac{4}{b} r^{2b} \log(\tfrac{1}{r}) & & \hspace{-1cm} \text{if $2b= \frac{1}{\alpha}$}, 	\\ \label{numultiplebounds}
c r^{\frac{1}{\alpha}} & \leq
 \nu(\partial \Omega \cap B_r(z_0))
\leq C r^{\frac{1}{\alpha}} & & \hspace{-1cm} \text{if $2b > \frac{1}{\alpha}$},
\end{align}
where $m_{\alpha}$ is the number of $\alpha_j$ such that $\alpha_j = \alpha$, i.e., the number of corners with the largest opening angle.
In particular, if $d\mu(z) \asymp |z-z_{0}|^{2b-2}d^{2}z$ as $z\to z_0$, then (\ref{nuestimate}) holds.
\end{theorem}

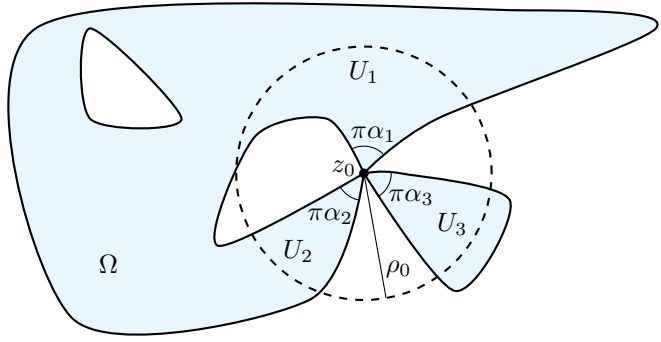
\begin{figure}
\begin{center}
\begin{tikzpicture}[scale = 2.4]
  \coordinate (origin) at (0, 0);

  \coordinate (z0) at (0,0);
  \coordinate (c2) at (0.4, 0.3);
  \coordinate (c3) at (1.6, 0.8);
  \coordinate (c4) at (0.9, 0.9);
  \coordinate (c5) at (-1.8, 0.8);
  \coordinate (c52) at (-1.6, -0.8);
  \coordinate (c53) at (-0.3, -0.7);
  \coordinate (c61) at (-0.8, -0.4);
  \coordinate (c62) at (-0.6, 0.2);
  \coordinate (c6) at (-0.2, 0.3);
  \coordinate (c7) at (0.4, -0.45);

 \draw[black, thick, fill=cyan!8] plot [smooth] coordinates {(z0) (c2) (c3) (c4) (c5) (c52) (c53) (z0)};
 \draw[black, thick, fill=white!8] plot [smooth] coordinates {(z0) (c61) (c62) (c6) (z0)};
 
\coordinate (d0) at (-1, 0.3);
\coordinate (d1) at (-1.5, 0.3);
\coordinate (d2) at (-1.5, 0.8);

\coordinate (e1) at (0.5, -0.65);
\coordinate (e2) at (0.8, -0.15);
\coordinate (e3) at (0.2, 0);
\draw[black, thick, fill=cyan!8] plot [smooth] coordinates {(z0) (e1) (e2) (e3) (z0)};

 \draw[black, thick, fill=white!8] plot [smooth cycle] coordinates {(d0) (d1) (d2)};

\node at (-1.4,-0.5){$\Omega$};


\fill (origin) circle (0.75pt);
\node at (-0.11,0.03){$z_0$};

\draw[black] (z0)+(46:0.15) arc (46:118:0.15);
\node at (76:0.22){$\pi \alpha_{1}$};
  
\draw[black] (z0)+(-103:0.15) arc (-103:-148:0.15);
\node at (-128:0.3){$\pi \alpha_{2}$};

\draw[black] (z0)+(4:0.15) arc (4:-57:0.15);
\node at (-27:0.29){$\pi \alpha_{3}$};

\draw[dashed, thick] (z0) circle (0.7cm);
\draw (z0)--(-80:0.7);
\node at (0.19,-0.53) {$\rho_{0}$};
\node at (0,0.56) {$U_{1}$};
\node at (-130:0.56) {$U_{2}$};
\node at (-30:0.56) {$U_{3}$};
\end{tikzpicture} 
 \caption{\label{figure:multicorner} Illustration of an open set $\Omega$ with three corners at $z_0$.}
 \end{center}
\end{figure}

The proof of Theorem \ref{mainth2} is presented in Section \ref{proofsec2}. The main idea of the proof is to show that the various corners decouple up to terms of $O(r^{1/\alpha})$ and then apply Theorem \ref{mainth} to each of the corners.


\section{Local structure of $\Omega$ near $z_0$}\label{Omegasec}

We consider the structure of $\Omega$ near the corner at $z_0$. We first treat the case of a single corner.

\subsection{A single corner}
Fix $\alpha \in (0, 2]$ and let $\Omega$ be a finitely connected Jordan domain in $\C^*$ such that $\Omega$ has a H\"older-$C^1$ corner of opening $\pi \alpha$ at $z_0 \in \C$. 
After applying a translation and a rigid rotation, we may assume that $z_0 = 0$ and that 
\begin{align}\label{argz0pialpha}
\arg z \to \begin{cases} 0 & \text{as $C_+ \ni z \to 0$}, \\
\pi \alpha & \text{as $C_- \ni z \to 0$},
\end{cases}
\end{align}
where $C_\pm \subset \partial \Omega$ are curves of class $C^{1,\gamma}$ for some $\gamma>0$ ending at $0$ and lying on different sides of $0$.

The next lemma shows that, for sufficiently small $\rho > 0$, $\Omega \cap B_\rho(0)$ is a small deformation of $S_\rho$, where
\begin{align}\label{Srhodef}
S_\rho := \{z \in \C : 0 < \arg z < \pi \alpha, \; 0 < |z| < \rho\}
\end{align}
is a circular sector of angle $\pi \alpha$ and radius $\rho$.

\begin{lemma}\label{nearcornerlemma}
There exists a radius $\rho_0 > 0$ such that the following hold, possibly after shortening the curves $C_\pm$:
\begin{enumerate}[$(i)$]

\item $\overline{B_{\rho_0}(0)} \cap \partial \Omega= C_+ \cup C_-$.

\item $C_+ \cap C_- = \{0\}$.

\item $C_+ = w_+([0, \rho_0])$ and $C_- = w_-([0, \rho_0])$ where the curves $w_\pm:[0, \rho_0] \to \C$ are $C^{1}$ on $(0,\rho_{0})$ and such that $|w_\pm(r)| = r$ for $r \in [0, \rho_0]$, i.e., $w_\pm$ are parametrizations of $C_\pm$ with the distance to the origin as parameter. Moreover, 
\begin{subequations}\label{argwplusminus}
\begin{align}\label{argwplus}
&\arg w_+(r) = O(r^{\gamma}) && \text{as $r \to 0$},
	\\ \label{argwminus}
&\arg w_-(r) = \pi \alpha + O(r^{\gamma})  && \text{as $r \to 0$},
\end{align}
\end{subequations}
where the branch is chosen so that $\arg(\cdot)$ is continuous in $(\bar{\Omega} \cap B_{\rho_{0}}(0))\setminus \{0\}$, and
\begin{subequations}\label{wpm der near 0}
\begin{align}
& w_{+}'(r) = 1+O(r^{\gamma}) & & \mbox{as } r \to 0, \\
& w_{-}'(r) = e^{i \pi \alpha}+O(r^{\gamma}) & & \mbox{as } r \to 0.
\end{align}
\end{subequations}
\item For each $0 < r \leq \rho_0$, it holds that $\Omega \cap \{|z| = r\} = J_r$ where $J_r$ is the circular arc
\begin{align}\label{Jrdef}
J_r := \{r e^{i\theta} : \arg w_+(r) < \theta < \arg  w_-(r)\}.
\end{align}

\item $J_r$ has length $r\Theta(r)$ where
\begin{align}\label{Thetaestimate}
\Theta(r) \leq \pi \alpha(1 + C_1 r^{\gamma}), \qquad 0 < r \leq \rho_0,
\end{align}
for some constant $C_1 > 0$.
\end{enumerate}

\end{lemma}
\begin{proof}
Since $C_+$ is $C^{1,\gamma}$ and tangent at $0$ to the ray $\{x \geq 0\}$, we can write $C_+$ as 
\begin{align*}
C_+ = \{ x(t) + i y(t) : 0 \leq t \leq t_1\}, 
\end{align*}
where $t_1>0$, and with $x,y \in C^{1,\gamma}([0,t_{1}])$ satisfying 
$x’(t) = x’(0) + O(t^\gamma)$, $y’(t) = O(t^\gamma)$, and $x’(0)>0$.
Rescaling $t$ if necessary, we may assume that $x’(0)=1$.
Integration of $x’(t) = 1 + O(t^\gamma)$ and $y’(t) = O(t^\gamma)$ then gives
$x(t) = t(1 + O(t^{\gamma}))$ and $y(t) = O(t^{1+\gamma})$.
Defining
\begin{align*}
r_+(t) := |x(t) + i y(t)| = x(t) \sqrt{ 1 + \left( \frac{y(t)}{x(t)} \right)^2 },
\end{align*}
we have
$$r_+’(t) = \frac{x'(t)+\frac{y(t) y'(t)}{x(t)}}{\sqrt{1 + \frac{y(t)^2}{x(t)^2}}} = 1+O(t^{\gamma})$$ 
as $t\to 0$. Hence, shrinking $t_1>0$ if necessary, we may assume that $r_+(t)$ is strictly increasing for $t \in [0,t_1]$. Let $t(r)$ for $r \in [0, r_1]$, where $r_1  = r_+(t_1)$, be the inverse of this function. The parametrization
\begin{align*}
w_+(r) := x(t(r)) + i y(t(r)), \quad r \in [0, r_1],
\end{align*}
satisfies $|w_+(r)| = r$ for all $r\in [0, r_1]$. Moreover, 
\begin{align*}
& \arg w_+(r) = \arctan \frac{y(t(r))}{x(t(r))} = O(t(r)^\gamma) = O(r^\gamma), & & r \to 0, \\
& w_{+}'(r) = x'(t(r))t'(r) + i y'(t(r))t'(r) = 1+O(r^{\gamma}), & & r \to 0.
\end{align*}
This completes the proof of the existence of a parametrization $w_{+}$ satisfying assertion $(iii)$; the proof for $w_{-}$ is similar and we omit it. Moreover, $\partial \Omega \setminus \big(w_+([0,\rho_0)) \cup w_-([0,\rho_0))\big)$ is a closed set in $\C^*$ which does not contain $0$. Therefore, shrinking $\rho_0$ if necessary, we may assume that $\partial \Omega \cap \overline{B_{\rho_0}(0)} = C_+ \cup C_-$ and that $C_+ \cap C_- = \{0\}$. Furthermore, for each $0 < r < \rho_0$, $\Omega \cap \{|z| = r\} = J_r$ where $J_r$ is defined by (\ref{Jrdef}). This completes the proof of assertions $(i)$-$(iv)$. 
From (\ref{argwplusminus}), we infer that the circular arc $J_r$ has 
length $r\Theta(r)$ where $\Theta(r) = \pi \alpha + O(r^{\gamma})$ as $r \to 0$, so shrinking $\rho_0$ if necessary, we obtain also $(v)$.
\end{proof}

\subsection{Multiple corners}\label{multiplestructuresubsec}

\begin{figure}
\begin{center}
\begin{tikzpicture}[scale = 2.4]
  \coordinate (origin) at (0, 0);

  \coordinate (z0) at (0,0);
  \coordinate (c2) at (0.4, 0.3);
  \coordinate (c3) at (1.6, 0.8);
  \coordinate (c4) at (0.9, 0.9);
  \coordinate (c5) at (-1.8, 0.8);
  \coordinate (c52) at (-1.6, -0.8);
  \coordinate (c53) at (-0.3, -0.7);
  \coordinate (c61) at (-0.8, -0.4);
  \coordinate (c62) at (-0.6, 0.2);
  \coordinate (c6) at (-0.2, 0.3);
  \coordinate (c7) at (0.4, -0.45);

 \draw[black, thick, fill=cyan!8] plot [smooth] coordinates {(z0) (c2) (c3) (c4) (c5) (c52) (c53) (z0)};
 \draw[black, thick, fill=white!8] plot [smooth] coordinates {(z0) (c61) (c62) (c6) (z0)};
 
\coordinate (d0) at (-1, 0.3);
\coordinate (d1) at (-1.5, 0.3);
\coordinate (d2) at (-1.5, 0.8);

\coordinate (e1) at (0.5, -0.65);
\coordinate (e2) at (0.8, -0.15);
\coordinate (e3) at (0.2, 0);
\draw[black, thick, fill=cyan!8] plot [smooth] coordinates {(z0) (e1) (e2) (e3) (z0)};

 \draw[black, thick, fill=white!8] plot [smooth cycle] coordinates {(d0) (d1) (d2)};

\fill (origin) circle (0.75pt);
\node at (-0.09,0.03){$0$};

\node at (-1.4,-0.5){$\Omega$};

\draw[black] (z0)+(46:0.15) arc (46:118:0.15);
\node at (76:0.22){$\pi \alpha_{1}$};
  
\draw[black] (z0)+(-103:0.15) arc (-103:-148:0.15);
\node at (-128:0.3){$\pi \alpha_{2}$};

\draw[black] (z0)+(4:0.15) arc (4:-57:0.15);
\node at (-27:0.29){$\pi \alpha_{3}$};

\draw[dashed, thick] (z0) circle (0.7cm);
\node at (-90:0.8) {$\partial B_{\rho_{0}}(0)$};
\node at (0,0.56) {$U_{1}$};
\node at (-173:0.56) {$V_{1}$};
\node at (-130:0.56) {$U_{2}$};
\node at (-80:0.56) {$V_{2}$};
\node at (-30:0.56) {$U_{3}$};
\node at (10:0.42) {$V_{0}\equiv V_{3}$};

\end{tikzpicture} 
 \caption{\label{figure:multicorner2} Illustration of Section \ref{multiplestructuresubsec} with $m=3$.}
 \end{center}
\end{figure}
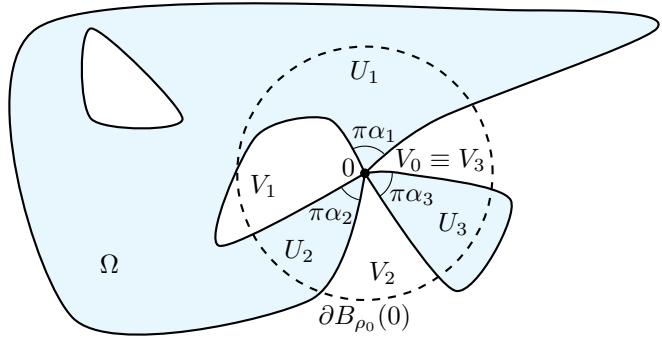

We next consider the case when $\Omega$ has $m \geq 1$ corners at $z_0 = 0$. 
Suppose $\Omega$ is an open subset of $\C^*$ satisfying $(i)$ and $(ii)$ of Theorem \ref{mainth2} with $z_0 = 0$. By applying Lemma \ref{nearcornerlemma} to each of the components $U_j$ of $\Omega \cap B_{\rho_0}(0)$, we see that after shrinking $\rho_0$ if necessary, the following hold (see Figure \ref{figure:multicorner2}):
\begin{enumerate}[$(i)$]
\item $\Omega \cap B_{\rho_0}(0) = \sqcup_{j=1}^m U_j$ and $\bar{\Omega} \cap \overline{B_{\rho_0}(0)} = \cup_{j=1}^m \bar{U}_j$ where $U_j$ has a H\"older-$C^1$ corner of opening $\pi \alpha_j \in (0,2\pi]$ at $0$.

\item $\overline{B_{\rho_0}(0)} \setminus \Omega = \cup_{j=1}^m \bar{V}_j$ where $\{V_j\}_1^m$ are the $m$ connected components of $B_{\rho_0}(0) \setminus \bar{\Omega}$.

\item The sets $U_j$ and $V_j$ are ordered so that $C_{j,+} := \bar{U}_j \cap \bar{V}_{j-1}$ ($V_{0} \equiv V_m$) and $C_{j,-} := \bar{U}_j \cap \bar{V}_{j}$ are Jordan arcs of class $C^{1,\gamma_j}$, $0 < \gamma_j \leq 1$, with $C_{j,+} \cap C_{j,-} = \{0\}$, and such that if $C_{j,\pm}$ are oriented outwards, then $U_j$ lies to the left of $C_{j,+}$ and to the right of $C_{j,-}$ for $j = 1, \dots, m$.

\item For each $j = 1, \dots, m$, $J_{j,r} := U_j \cap \partial B_r(0)$ is an arc of length $r\Theta_j(r)$ where
\begin{align}\label{Thetajestimate}
\Theta_j(r) \leq \pi \alpha_j(1 + C_j r^{\gamma_j}), \qquad 0 < r \leq \rho_0,
\end{align}
for some constant $C_j > 0$. 
\end{enumerate}
We henceforth assume that $\rho_0 >0$ has been chosen so small that the above properties hold.

\section{Estimates of harmonic measure}\label{estimatesec}
In this section, we derive upper bounds on $\omega(z, \partial U_j \cap B_{r}(0), \Omega)$ that will be used in the proofs of Theorem \ref{mainth} and Theorem \ref{mainth2}.

Let $\Omega$ be an open connected subset of $\C^*$. A {\it metric} $\rho$ on $\Omega$ is a non-negative Borel measurable function $\rho$ on $\Omega$ such that the $\rho$-area of $\Omega$,
$$A(\Omega, \rho) = \int_\Omega \rho^2 d^2z$$
satisfies $0 < A(\Omega, \rho) < \infty$. 
Let $E$ and $F$ be subsets of $\bar{\Omega}$, and let $\Gamma$ be the family of all connected arcs in $\Omega$ joining $E$ and $F$. The {\it extremal distance} $d_\Omega(E,F)$ from $E$ to $F$ is defined by
\begin{align}\label{extremaldistancedef}
d_\Omega(E,F) = \sup_\rho \frac{L(\Gamma, \rho)^2}{A(\Omega, \rho)}, \qquad \text{where $L(\Gamma, \rho) = \inf_{\gamma \in \Gamma} \int_\gamma \rho |dz|$.}
\end{align}

We will need the following result which is Theorem H.7 in \cite{GM2005}.

\begin{lemma}{\upshape \cite[Theorem H.7]{GM2005}}\label{lemmaH7}
Let $\tilde{\Omega}$ be a finitely connected Jordan domain, and let $E$ be a finite union of arcs contained in one component $\tilde{\Gamma}$ of $\partial \tilde{\Omega}$. Suppose $\sigma$ is a Jordan arc in $\C$ connecting $z_1 \in \tilde{\Omega}$ to $\tilde{\Gamma} \setminus E$. Then
$$\omega(z_1, E, \tilde{\Omega}) \leq \frac{8}{\pi} e^{-\pi d_{\tilde{\Omega}\setminus \sigma}(\sigma, E)}.$$
\end{lemma}

The next lemma will be used to obtain upper bounds on $\omega(z, \partial U_j \cap B_{r}(0), \Omega)$. The lemma treats the case when $\Omega$ has any finite number of corners at $z_0 = 0$; the case of a single corner is included as a special case. 
The proof is basically a combination of the proofs of \cite[Theorem IV.6.2]{GM2005} and \cite[Theorem H.8]{GM2005}.
\cite[Theorem H.8]{GM2005} treats the cartesian case whereas we are interested in the polar case. 
In \cite[Theorem IV.6.2]{GM2005}, the estimate is for $\omega(z_1, E, \Omega)$ with $|z_1|$ small and $E$ outside a large disk. We need the opposite situation: $E$ inside a small disk and $|z_1|$ large. We therefore provide a proof.
We assume that $\rho_0$ is chosen as in Section \ref{multiplestructuresubsec}.

\begin{lemma}\label{extremaldistancelemma}
Suppose $\Omega$ is an open subset of $\C^*$ fulfilling $(i)$ and $(ii)$ of Theorem \ref{mainth2} with $z_0 = 0$ and some integer $m \geq 1$. Fix $j \in \{1, \dots, m\}$ and let $0 < r_0 < R_0 \leq \rho_0$. Let $z_1 \in \Omega$ be such that $|z_1| \geq R_0$. If $r \Theta_j(r)$ is the length of $U_j \cap \partial B_r(0)$, then
\begin{align}\label{omegaleq8pi}
\omega(z_1, \partial U_j \cap B_{r_0}(0), \Omega) \leq \frac{8}{\pi} e^{-\pi \int_{r_0}^{R_0} \frac{dr}{r \Theta_j(r)}}.
\end{align}
\end{lemma}
\begin{proof}
If $m \leq 2$, let $\Omega' := \Omega$. If $m \geq 3$, let $\Omega' := \Omega \cup \big(B_{\rho_0}(0) \cap \cup_{i \neq j-1, j} \bar{V}_i \big) \setminus \{0\}$, where the union is over all $i = 1, \dots, m$ with $i \neq j-1$ and $i \neq j$, i.e., over all $i$ for which $V_i$ is not adjacent to $U_j$. 
Let $E := U_j \cap \partial B_{r_0}(0)$ and let $\tilde{\Omega}$ be the component of $\Omega' \setminus E$ that contains $z_1$. 
The definition of $\Omega'$ implies that $\tilde{\Omega}$ has at most a single corner at $0$; in particular, $\tilde{\Omega}$ is a finitely connected Jordan domain.

Let $\tilde{\Gamma}$ be the component of $\partial \tilde{\Omega}$ containing $E$, see Figure \ref{figure:tildeEfigure}. Note that $E$ separates $\partial U_j \cap B_{r_0}(0)$ from $z_1$ in $\Omega'$. Hence, by the maximum principle,  
$$\omega(z_1, \partial U_j \cap B_{r_0}(0), \Omega) \leq \omega(z_1, \partial U_j \cap B_{r_0}(0), \Omega') \leq \omega(z_1, E, \tilde{\Omega}).$$
 
Let $\sigma \subset \{z : |z| = |z_1|\}$ be a curve (not necessarily contained in $\Omega$) connecting $z_1$ to $\tilde{\Gamma}$. 
By Lemma \ref{lemmaH7},
\begin{align*}
\omega(z_1, E, \tilde{\Omega}) \leq \frac{8}{\pi} e^{-\pi d_{\tilde{\Omega}\setminus \sigma}(\sigma, E)}.
\end{align*}
Therefore it only remains to show that
\begin{align}\label{dgeqint}
d_{\tilde{\Omega}\setminus \sigma}(\sigma, E) \geq \int_{r_0}^{R_0} \frac{dr}{r\Theta(r)}.
\end{align}

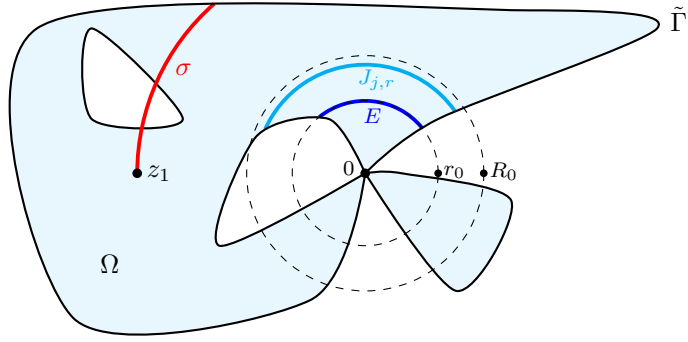
\begin{figure}
\begin{center}
\begin{tikzpicture}[scale = 2.4]
  \coordinate (origin) at (0, 0);

  \coordinate (z0) at (0,0);
  \coordinate (c2) at (0.4, 0.3);
  \coordinate (c3) at (1.6, 0.8);
  \coordinate (c4) at (0.9, 0.9);
  \coordinate (c5) at (-1.8, 0.8);
  \coordinate (c52) at (-1.6, -0.8);
  \coordinate (c53) at (-0.3, -0.7);
  \coordinate (c61) at (-0.8, -0.4);
  \coordinate (c62) at (-0.6, 0.2);
  \coordinate (c6) at (-0.2, 0.3);
  \coordinate (c7) at (0.4, -0.45);

 \draw[white, thick, fill=cyan!8] plot [smooth] coordinates {(z0) (c2) (c3) (c4) (c5) (c52) (c53) (z0)};
 \draw[white, thick, fill=white!8] plot [smooth] coordinates {(z0) (c61) (c62) (c6) (z0)};
 
\coordinate (d0) at (-1, 0.3);
\coordinate (d1) at (-1.5, 0.3);
\coordinate (d2) at (-1.5, 0.8);

\coordinate (e1) at (0.5, -0.65);
\coordinate (e2) at (0.8, -0.15);
\coordinate (e3) at (0.2, 0);
\draw[black, thick, fill=cyan!8] plot [smooth] coordinates {(z0) (e1) (e2) (e3) (z0)};

 \draw[black, thick, fill=white!8] plot [smooth cycle] coordinates {(d0) (d1) (d2)};

\fill (origin) circle (0.75pt);
\node at (-0.09,0.03){\footnotesize $0$};

\node at (-1.4,-0.5){$\Omega$};

\draw[line width=0.5mm, blue] (z0)+(39:0.4) arc (39:129:0.4);
\node at (83:0.32){\footnotesize \textcolor{blue}{$E$}};

\draw[line width=0.5mm, cyan] (z0)+(35:0.6) arc (35:156:0.6);
\node at (83:0.52){\footnotesize \textcolor{cyan}{$J_{j,r}$}};
%
%


\draw[dashed] (z0) circle (0.65cm);
\fill (0.65,0) circle (0.6pt); 
\node at (0.76,0.01){\footnotesize $R_{0}$};

\draw[dashed] (z0) circle (0.4cm);
\fill (0.4,0) circle (0.6pt); 
\node at (0.49,0.01){\footnotesize $r_{0}$};


\draw[line width=0.5mm, red] (z0)+(180:1.25) arc (180:131.5:1.25);
\node at (150:1.15){\textcolor{red}{$\sigma$}};

\node at (1.72,0.85) {$\tilde{\Gamma}$};
\node at (180:1.13) {$z_{1}$};
\fill (180:1.25) circle (0.75pt);

 \draw[black, thick] plot [smooth] coordinates {(z0) (c2) (c3) (c4) (c5) (c52) (c53) (z0)};
 \draw[black, thick] plot [smooth] coordinates {(z0) (c61) (c62) (c6) (z0)};
\end{tikzpicture} 
 \caption{\label{figure:tildeEfigure} Illustration of Lemma \ref{extremaldistancelemma}.}
 \end{center}
\end{figure}

Let $J_{j,r} := U_j \cap \partial B_r(0)$.
For $z \in \tilde{\Omega}\setminus \sigma$, we define the metric
\begin{align*}
\rho(z) = \begin{cases} \frac{1}{r\Theta_j(r)} & \text{for $z \in J_{j,r}$}, \\
0 & \text{for $z \in \tilde{\Omega}\setminus \cup_{r_0 < r < R_0} J_{j,r}$},
\end{cases}
\end{align*}
where $r = |z|$. 
Let $\Gamma$ be the family of all arcs in $\tilde{\Omega}\setminus \sigma$ connecting $E$ and $\sigma$. For each $r \in (r_0, R_0)$, $J_{j,r}$ separates $\sigma$ from $E$ in $\tilde{\Omega}$. Hence, if $\gamma \in \Gamma$, then
$$\int_\gamma \rho(z) |dz| \geq \int_{r_0}^{R_0} \frac{dr}{r\Theta_j(r)},$$
and so
$$L(\Gamma, \rho) \geq \int_{r_0}^{R_0} \frac{dr}{r\Theta_j(r)}.$$
Furthermore, the $\rho$-area of $\tilde{\Omega}\setminus \sigma$ is given by
$$A(\tilde{\Omega}\setminus \sigma, \rho) =   \int_{r_0}^{R_0} \int_{\{\theta:re^{i\theta}\in J_{j,r}\}} \frac{1}{r^2\Theta_j(r)^2} r d\theta dr   
= \int_{r_0}^{R_0} \frac{dr}{r\Theta_j(r)}.$$
Hence, in view of (\ref{extremaldistancedef}),
\begin{align*}
d_{\tilde{\Omega}\setminus \sigma}(\sigma, E) \geq \frac{L(\Gamma, \rho)^2}{A(\tilde{\Omega}\setminus \sigma, \rho)} \geq \int_{r_0}^{R_0} \frac{dr}{r\Theta_j(r)},
\end{align*}
which proves (\ref{dgeqint}) and thus completes the proof of the lemma.
\end{proof}

By applying Lemma \ref{extremaldistancelemma}, we can prove the next lemma which provides upper bounds on the harmonic measure $\omega(z, \partial U_j \cap B_r(0), \Omega)$ for any $0 < r < \min(\rho_{0},|z|)$.

\begin{figure}
\begin{center}
\begin{tikzpicture}[scale = 2.4]
  \coordinate (origin) at (0, 0);

  \coordinate (z0) at (0,0);
  \coordinate (c2) at (0.4, 0.3);
  \coordinate (c3) at (1.6, 0.8);
  \coordinate (c4) at (0.9, 0.9);
  \coordinate (c5) at (-1.8, 0.8);
  \coordinate (c52) at (-1.6, -0.8);
  \coordinate (c53) at (-0.3, -0.7);
  \coordinate (c61) at (-0.8, -0.4);
  \coordinate (c62) at (-0.6, 0.2);
  \coordinate (c6) at (-0.2, 0.3);
  \coordinate (c7) at (0.4, -0.45);

 \draw[white, thick, fill=cyan!8] plot [smooth] coordinates {(z0) (c2) (c3) (c4) (c5) (c52) (c53) (z0)};
 \draw[white, thick, fill=white!8] plot [smooth] coordinates {(z0) (c61) (c62) (c6) (z0)};
 
\coordinate (d0) at (-1, 0.3);
\coordinate (d1) at (-1.5, 0.3);
\coordinate (d2) at (-1.5, 0.8);

\coordinate (e1) at (0.5, -0.65);
\coordinate (e2) at (0.8, -0.15);
\coordinate (e3) at (0.2, 0);
\draw[black, thick, fill=cyan!8] plot [smooth] coordinates {(z0) (e1) (e2) (e3) (z0)};

 \draw[black, thick, fill=white!8] plot [smooth cycle] coordinates {(d0) (d1) (d2)};

\fill (origin) circle (0.75pt);
\node at (-0.09,0.03){\footnotesize $0$};

\node at (-1.4,-0.5){$\Omega$};

\draw[line width=0.5mm, blue] (z0)+(39:0.4) arc (39:129:0.4);
\node at (83:0.32){\footnotesize \textcolor{blue}{$E$}};

\draw[line width=0.5mm, cyan] (z0)+(35:0.6) arc (35:156:0.6);
\node at (83:0.52){\footnotesize \textcolor{cyan}{$J_{j,r}$}};
%
%


\draw[dashed] (z0) circle (0.7cm);
\fill (0.7,0) circle (0.6pt); 
\node at (0.81,0.01){\footnotesize $\rho_{0}$};

\draw[dashed] (z0) circle (0.4cm);
\fill (0.4,0) circle (0.6pt); 
\node at (0.49,0.01){\footnotesize $r_{0}$};

\node at (180:1.15) {$z$};
\fill (180:1.25) circle (0.75pt);

 \draw[black, thick] plot [smooth] coordinates {(z0) (c2) (c3) (c4) (c5) (c52) (c53) (z0)};
 \draw[black, thick] plot [smooth] coordinates {(z0) (c61) (c62) (c6) (z0)};
\end{tikzpicture} 
 \caption{\label{harmonicfigure} Illustration of the argument leading to (\ref{omegazEOmega}) in the case of $|z| \geq \rho_0$.}
 \end{center}
\end{figure}
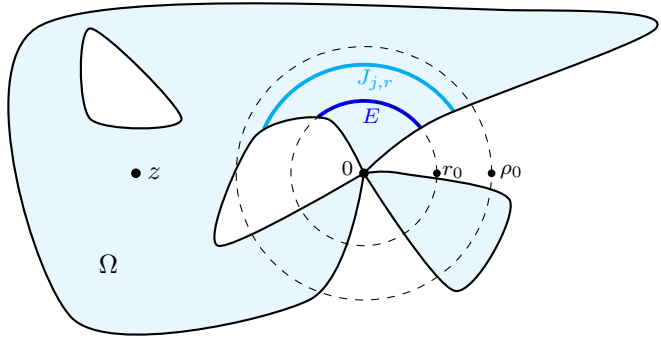

\begin{lemma}\label{omegaupperboundlemma}
Suppose $\Omega$ is an open subset of $\C^*$ fulfilling $(i)$ and $(ii)$ of Theorem \ref{mainth2} with $z_0 = 0$ and $m \geq 1$. Fix $j \in \{1, \dots, m\}$ and let $\alpha_j$, $C_j$, and $\gamma_j$ be as in (\ref{Thetajestimate}). If $r > 0$ and $z \in \Omega$ are such that $0 < r < |z| \leq \rho_0$, then
\begin{align}\label{omegaleq8pirzmultiple}
\omega(z, \partial U_j \cap B_{r}(0), \Omega) \leq \frac{8}{\pi} \bigg(\frac{r}{|z|}\bigg)^{\frac{1}{\alpha_j}} 
\left(1 +  C_j |z|^{\gamma_j }\right)^{\frac{1}{\alpha_j \gamma_j}}.
\end{align}
Moreover, if $r > 0$ and $z \in \Omega$ are such that $0 < r < \rho_0 \leq |z|$, then
\begin{align}\label{omegaleq8pirzmultiple2}
\omega(z, \partial U_j \cap B_{r}(0), \Omega) \leq \frac{8}{\pi} \bigg(\frac{r}{\rho_0}\bigg)^{\frac{1}{\alpha_j}} 
\left(1 +  C_j \rho_0^{\gamma_j }\right)^{\frac{1}{\alpha_j  \gamma_j}}.
\end{align}
\end{lemma}
\begin{proof}
Applying Lemma \ref{extremaldistancelemma} with $R_0 = \min(\rho_0, |z|)$, we find that if $r_0 \in [0, \rho_0)$, then (see Figure \ref{harmonicfigure})
\begin{align}\label{omegazEOmega}
\omega(z, \partial U_j \cap B_{r_0}(0), \Omega) \leq \frac{8}{\pi} e^{-\pi \int_{r_0}^{\min(\rho_0, |z|)} \frac{dr}{r \Theta_j(r)}} \quad \text{for all $z \in \Omega$ with $|z| > r_0$}.
\end{align}
Using (\ref{Thetajestimate}), we see that, whenever $z \in \Omega$ and $0 < r_0 < |z| \leq \rho_0$,
\begin{align*}
\int_{r_0}^{\min(\rho_0, |z|)} \frac{dr}{r \Theta_j(r)}
& \geq \int_{r_0}^{|z|} \frac{dr}{r \pi \alpha_j (1 + C_j r^{\gamma_j})}
= \frac{\log(\frac{|z|}{r_0})}{\pi  \alpha_j}
- \frac{\log \Big(\frac{1 +  C_j |z|^{\gamma_j }}{1 +  C_j r_0^{\gamma_j }}\Big)}{\pi  \alpha_j  \gamma_j }
	\\
& \geq \frac{\log(\frac{|z|}{r_0})}{\pi  \alpha_j}
- \frac{\log \left(1 +  C_j |z|^{\gamma_j }\right)}{\pi  \alpha_j  \gamma_j }.
\end{align*}
Employing this inequality in (\ref{omegazEOmega}) and replacing $r_0$ by $r$, we arrive at (\ref{omegaleq8pirzmultiple}). The proof of (\ref{omegaleq8pirzmultiple2}) follows similarly using that
$$\int_{r_0}^{\min(\rho_0, |z|)} \frac{dr}{r \Theta_j(r)}
\geq  \frac{\log(\frac{\rho_0}{r_0})}{\pi  \alpha_j}
- \frac{\log \left(1 +  C_j \rho_0^{\gamma_j }\right)}{\pi  \alpha_j  \gamma_j }
$$
whenever $z \in \Omega$ and $0 < r_{0} < \rho_0 \leq |z|$.
\end{proof}

\section{Proof of Theorem \ref{mainth}}\label{proofsec}
Let $\Omega$ be a finitely connected Jordan domain in $\C^*$ such that $\Omega$ has a H\"older-$C^1$ corner of opening $\pi \alpha$ at $0$.
Fix $b > 0$ and let $\mu$ be a non-negative measure on $\Omega$ of finite total mass such that $d\mu(z) = (1+o(1)) |z|^{2b-2}d^{2}z$ as $z\to 0$. Let $\nu = \mathrm{Bal}(\mu,\partial \Omega)$. Shrinking $\rho_0$ if necessary, we may assume that (\ref{muc0o1eps}) holds for all measurable subsets $A$ of $\Omega \cap \overline{B_{\rho_0}(0)}$.

Applying Lemma \ref{omegaupperboundlemma} with $m = 1$, we find the following estimates: If $r > 0$ and $z \in \Omega$ are such that $0 < r < |z| \leq \rho_0$, then
\begin{align}\label{omegaleq8pirz}
\omega(z, \partial \Omega \cap B_{r}(0), \Omega) \leq \frac{8}{\pi} \bigg(\frac{r}{|z|}\bigg)^{\frac{1}{\alpha}} 
\left(1 +  C_1 |z|^{\gamma }\right)^{\frac{1}{\alpha \gamma}},
\end{align}
while if $r > 0$ and $z \in \Omega$ are such that $0 < r < \rho_0 \leq |z|$, then
\begin{align}\label{omegaleq8pirz2}
\omega(z, \partial \Omega \cap B_{r}(0), \Omega) \leq \frac{8}{\pi} \bigg(\frac{r}{\rho_0}\bigg)^{\frac{1}{\alpha}} 
\left(1 +  C_1 \rho_0^{\gamma }\right)^{\frac{1}{\alpha \gamma}},
\end{align}
where $C_1 > 0$ and $0 < \gamma \leq 1$.

The proof of the following lemma is based on integration of the inequalities (\ref{omegaleq8pirz}) and (\ref{omegaleq8pirz2}).

\begin{lemma}\label{nuuppersinglelemma}
For every $\epsilon > 0$, we have
$$\nu(\partial \Omega \cap B_r(0)) \leq \begin{cases}
(1+\epsilon) \frac{\pi \alpha}{2b} (1  + \frac{16 b}{\pi (\frac{1}{\alpha} - 2b)})r^{2b} & \text{if $2b < \frac{1}{\alpha}$}, \\
(1+\epsilon) 8  \alpha r^{2b} \log(\frac{1}{r}) &  \text{if $2b= \frac{1}{\alpha}$}, \\
C r^{\frac{1}{\alpha}} &  \text{if $2b > \frac{1}{\alpha}$},
\end{cases}
$$
for all sufficiently small $r > 0$. 
\end{lemma}
\begin{proof}
Recall that, by definition, $\nu(\partial \Omega \cap B_r(0)) = \int_\Omega \omega(z, \partial \Omega \cap B_r(0), \Omega) d\mu(z)$. For $r \in (0, \rho_0)$ we write $\nu(\partial \Omega \cap B_r(0))$ as the sum of three integrals:
$$\nu(\partial \Omega \cap B_r(0)) = I_{1,r} + I_{2,r} + I_{3,r},$$
where
\begin{align*}
&  I_{1,r} := \int_{\Omega \cap \overline{B_{r}(0)}} \omega(z, \partial \Omega \cap B_r(0), \Omega) d\mu(z), 
&& I_{2,r} := \int_{(\Omega \cap B_{\rho_0}(0))\setminus \overline{B_{r}(0)}} \omega(z, \partial \Omega \cap B_r(0), \Omega) d\mu(z),
	\\
&
I_{3,r} := \int_{\Omega \setminus B_{\rho_0}(0)} \omega(z, \partial \Omega \cap B_r(0), \Omega) d\mu(z).
\end{align*}

Let $\epsilon > 0$. Shrinking $\rho_0$ if necessary, we may assume that 
\begin{align}\label{1C0rho01plusepsilon}
\left(1 +  C_1 \rho_0^{\gamma }\right)^{\frac{1}{\alpha  \gamma} + 1} \leq 1 + \epsilon.
\end{align}
Using the fact that the harmonic measure of any set is $\leq 1$,  (\ref{muc0o1eps}), and (\ref{Thetaestimate}), we find
\begin{align}\nonumber
I_{1,r} & \leq \int_{\Omega \cap \overline{B_{r}(0)}}  d\mu(z)
\leq (1+\epsilon) \int_{\Omega \cap \overline{B_{r}(0)}}  |z|^{2b-2}d^{2}z
\leq (1+\epsilon) 
\pi \alpha(1 + C_1 r^{\gamma}) \int_0^r  \rho^{2b-1} d\rho
	\\ \label{I1restimate}
& \leq (1+\epsilon)^2 \frac{\pi \alpha}{2b} r^{2b}
\end{align}
for all sufficiently small $r > 0$. 
To estimate $I_{2,r}$, we use (\ref{muc0o1eps}), (\ref{omegaleq8pirz}), and (\ref{Thetaestimate}) to write, for all sufficiently small $r > 0$,
\begin{align*}\nonumber
I_{2,r} 
& \leq \int_{(\Omega \cap B_{\rho_0}(0))\setminus \overline{B_{r}(0)}} \frac{8}{\pi} \bigg(\frac{r}{|z|}\bigg)^{\frac{1}{\alpha}} 
\left(1 +  C_1 |z|^{\gamma }\right)^{\frac{1}{\alpha  \gamma}}  (1+\epsilon) |z|^{2b-2}d^{2}z
	\\\nonumber
& \leq  \int_r^{\rho_0} \frac{8}{\pi} \bigg(\frac{r}{\rho}\bigg)^{\frac{1}{\alpha}} 
\left(1 +  C_1 \rho^{\gamma }\right)^{\frac{1}{\alpha  \gamma}} (1+\epsilon) \pi \alpha(1 + C_1 \rho^{\gamma})  \rho^{2b-1}d\rho
	\\  \nonumber
& \leq 8 (1+\epsilon) \alpha
\left(1 +  C_1 \rho_0^{\gamma }\right)^{\frac{1}{\alpha  \gamma} + 1} 
r^{\frac{1}{\alpha}} \int_r^{\rho_0} \rho^{2b-1- \frac{1}{\alpha}}d\rho.
\end{align*}
In light of (\ref{1C0rho01plusepsilon}), this gives	
\begin{align} \label{I2restimate}
I_{2,r} & \leq 
\begin{cases}
(1+\epsilon)^2 8 \alpha
 r^{\frac{1}{\alpha}} \frac{\rho_0^{2b - \frac{1}{\alpha}} - r^{2b - \frac{1}{\alpha}}}{2b - \frac{1}{\alpha}} 
\leq C r^{\frac{1}{\alpha}} + (1+\epsilon)^2 \frac{8 \alpha}{|2b - \frac{1}{\alpha}|} r^{2b} & \text{if $2b \neq \frac{1}{\alpha}$}, 
	\\
(1+\epsilon)^2 8 \alpha r^{2b} \log(\frac{\rho_0}{r})
& \text{if $2b = \frac{1}{\alpha}$},
\end{cases}
\end{align}
for all sufficiently small $r > 0$.
Finally, using (\ref{omegaleq8pirz2}) and the fact that $\mu$ has finite total mass, we obtain
\begin{align}\label{I3restimate}
I_{3,r} 
\leq  \frac{8}{\pi} \bigg(\frac{r}{\rho_0}\bigg)^{\frac{1}{\alpha}} 
\left(1 +  C_1 \rho_0^{\gamma }\right)^{\frac{1}{\alpha  \gamma}} \int_{\Omega \setminus B_{\rho_0}(0)} d\mu(z)
\leq C r^{\frac{1}{\alpha}}
\end{align}
for all sufficiently small $r > 0$.
Since $\nu(\partial \Omega \cap B_r(0)) = I_{1,r} + I_{2,r} + I_{3,r}$ and $\epsilon >0$ was arbitrary, the desired conclusion follows from (\ref{I1restimate}), (\ref{I2restimate}), and (\ref{I3restimate}).
\end{proof}

Lemma \ref{nuuppersinglelemma} establishes the upper bounds on $\nu(\partial \Omega \cap B_r(0))$ stated in (\ref{nubounds}). 
In what follows, we establish the lower bounds on $\nu(\partial \Omega \cap B_r(0))$ stated in (\ref{nubounds}). 

Let $\Gamma_1$ be the component of $\partial \Omega$ containing $0$. Let $\Omega_1 \supset \Omega$ be the component of $\C^* \setminus \Gamma_1$ that contains $\Omega$. Then $\partial \Omega_1 = \Gamma_1$ is a Jordan curve and $\Omega_1$ has a H\"older-$C^1$ corner of opening  $\pi \alpha$ at $0$. By the Jordan curve theorem and the Riemann mapping theorem, there is a conformal map $f$ of the open upper half-plane $\mathbb{H}$ onto $\Omega_1$ such that $f(0) = 0$, $C_+ \subset f(\R_+)$, and $C_- \subset f(\R_-)$. Decreasing $\gamma$ if necessary, in what follows we assume that $\gamma < \frac{1}{2}$. By Lemma \ref{exerciselemma}, since $0$ is a H\"older-$C^1$ corner, $f(z) = c z^\alpha(1 + O(z^{\alpha \gamma}))$ as $z \to 0$, $z \in  \bar{\mathbb{H}}$, where $c \in \C \setminus \{0\}$ is a constant and $z^\alpha:=|z|^{\alpha}e^{i\alpha \arg z}$, $\arg z \in (-\frac{\pi}{2},\frac{3\pi}{2})$. From (\ref{argz0pialpha}), we infer that $\arg c = 0$, so that replacing $f(z)$ with $f(c^{-1/\alpha} z)$, we have
\begin{align}\label{fnear0}
f(z) = z^\alpha(1 + O(z^{\alpha \gamma})) \qquad \text{as $z \to 0$, $z \in  \bar{\mathbb{H}}$}.
\end{align}

Let $a$ be so small that the image under $f$ of the half-disk $B_a(0) \cap \mathbb{H}$ of radius $a$ is contained in $\Omega \cap B_{\rho_0}(0)$, where $\rho_0$ is as in Lemma \ref{nearcornerlemma}.
Let $S_{a^{\alpha}}$ be the circular sector of angle $\pi \alpha$ and radius $a^{\alpha}$ defined by (\ref{Srhodef}).
Let $h$ be the conformal map of the half-disk $B_a(0) \cap \mathbb{H}$ onto $S_{a^\alpha}$ given by $h(z) = z^\alpha$. 
Then $f \circ h^{-1}$ is a conformal map of $S_{a^\alpha}$ onto $f(B_a(0) \cap \mathbb{H})$. 

\begin{lemma}
For every $\epsilon > 0$, there exist an $a > 0$ such that 
\begin{align}\label{nugeqc1c3}
\nu(\partial \Omega \cap B_r(0)) 
\geq (1-\epsilon) \int_{S_{a^\alpha}} \omega((f \circ h^{-1})(w), \partial \Omega \cap B_r(0), f(B_a(0) \cap \mathbb{H})) |w|^{2b-2} d^{2}w
\end{align}
for all sufficiently small $ r > 0$.
\end{lemma}
\begin{proof}
Since $f(B_a(0) \cap \mathbb{H}) \subset \Omega$, we have
$$\nu(\partial \Omega \cap B_r(0)) = \int_\Omega \omega(z, \partial \Omega \cap B_r(0), \Omega) d\mu(z)
\geq 
\int_{f(B_a(0) \cap \mathbb{H})} \omega(z, \partial \Omega \cap B_r(0), \Omega) d\mu(z).$$ 
Utilizing that $f(B_a(0) \cap \mathbb{H})$ is contained in $\Omega \cap B_{\rho_0}(0)$ and (\ref{muc0o1eps}), we obtain
\begin{align}\label{nupartialOmegacapBr}
\nu(\partial \Omega \cap B_r(0)) 
\geq (1-\epsilon) \int_{f(B_a(0) \cap \mathbb{H})} \omega(z, \partial \Omega \cap B_r(0), \Omega) |z|^{2b-2}d^{2}z
\end{align}
for all sufficiently small $ r > 0$.
Whenever $r$ is so small that $\partial \Omega \cap B_r(0)$ is a subset of $\partial f(B_a(0) \cap \mathbb{H})$, the maximum principle yields
$$\omega(z, \partial \Omega \cap B_r(0), \Omega) \geq \omega(z, \partial \Omega \cap B_r(0), f(B_a(0) \cap \mathbb{H}))$$
and thus (\ref{nupartialOmegacapBr}) implies
\begin{align}\label{nupartialOmegaBr}
\nu(\partial \Omega \cap B_r(0)) 
\geq (1-\epsilon) \int_{f(B_a(0) \cap \mathbb{H})} \omega(z, \partial \Omega \cap B_r(0), f(B_a(0) \cap \mathbb{H})) |z|^{2b-2}d^{2}z.
\end{align}
The final step is to perform the change of variables $z = (f \circ h^{-1})(w) = f(w^{\frac{1}{\alpha}})$ in the integral in (\ref{nupartialOmegaBr}). It follows from (\ref{fnear0}) that
\begin{align}\label{zexpansionw}
z = w (1 + O(w^{\gamma})) \qquad \text{as $w \to 0$, $w \in \bar{S}_{a^\alpha}$}.
\end{align}
Moreover, by \cite[Theorem 3.9]{P1992},
$$f'(w^{\frac{1}{\alpha}}) = \alpha w^{1-\frac{1}{\alpha}} (1 + o(1))  \qquad \text{as $w \to 0$, $w \in \bar{S}_{a^\alpha}$}.$$
Therefore, shrinking $a$ if necessary, we have
$$|z|^{2b-2}d^{2}z = |z|^{2b-2} |f'(w^{\frac{1}{\alpha}})|^2 \bigg|\frac{w^{\frac{1}{\alpha} - 1}}{\alpha}\bigg|^2 d^{2}w
\geq (1-\epsilon) |w|^{2b-2} d^{2}w$$
for all $z \in f(B_a(0) \cap \mathbb{H})$.
Hence, changing variables from $z$ to $w$ in (\ref{nupartialOmegaBr}) and recalling that $\epsilon > 0$ was arbitrary, we conclude that (\ref{nugeqc1c3}) holds.
\end{proof}

Let $\mu_{b}$ be the restriction of the measure $|w|^{2b-2} d^{2}w$ to $S_{a^\alpha}$. Let $\nu_{b} := \mathrm{Bal}(\mu_{b}, \partial S_{a^\alpha})$ be the balayage of $\mu_{b}$ onto $\partial S_{a^\alpha}$ so that
\begin{align*}
\nu_{b}(E) = \int_{S_{a^\alpha}} \omega(w, E, S_{a^\alpha}) |w|^{2b-2} d^{2}w \qquad \text{for Borel subsets $E$ of $\partial S_{a^\alpha}$}.
\end{align*}
By the conformal invariance of harmonic measure, we have
$$\omega((f \circ h^{-1})(w), \partial \Omega \cap B_r(0), f(B_a(0) \cap \mathbb{H})) = \omega(w, (h \circ f^{-1})(\partial \Omega \cap B_r(0)), S_{a^\alpha}).$$
Moreover, by (\ref{zexpansionw}), the set $(h \circ f^{-1})(\partial \Omega \cap B_r(0))$ contains $\partial S_{a^\alpha} \cap B_{(1-\epsilon) r}(0)$ for all small enough $r > 0$. Consequently, we deduce from (\ref{nugeqc1c3}) that, for all sufficiently small $r > 0$,
\begin{align}\label{nupartialOmegaBr0}
\nu(\partial \Omega \cap B_r(0)) 
\geq (1-\epsilon) \nu_{b}\big((h \circ f^{-1})(\partial \Omega \cap B_r(0))\big)
\geq (1-\epsilon) \nu_{b}\big(\partial S_{a^\alpha} \cap B_{(1-\epsilon) r}(0)\big).
\end{align}

On the other hand, by \cite[Remark 2.18]{C2023}, we have, for $R \in (0, a^{\alpha})$, 
\begin{align*}
& \nu_{b}\big(\partial S_{a^\alpha} \cap B_{R}(0)\big) =
	\\
&\hspace{1.4cm} \begin{cases}
2\int_0^R \frac{4 a^{\alpha 2b}}{\alpha \pi r} \sum_{j=0}^\infty \frac{(\frac{r}{a^\alpha})^{2b} - (\frac{r}{a^\alpha})^{\frac{2j}{\alpha} + \frac{1}{\alpha}}}{(\frac{2j}{\alpha} + \frac{1}{\alpha})^2 - (2b)^2}  dr
& \text{if $2b \notin \frac{1}{\alpha}+\frac{2}{\alpha}\N_{\geq 0}$}, \\
2\int_0^R \frac{1}{\pi} \Big(2 r^{2b -1} \log(\frac{a^\alpha}{r}) + \sum_{j=1}^\infty \alpha r^{2b - 1}\frac{1- (\frac{r}{a^\alpha})^{\frac{2j}{\alpha}}}{j(j+1)} \Big) dr
&  \text{if $2b= \frac{1}{\alpha}$}, \\
 2\int_0^R \frac{1}{\pi} \Big(\frac{2 r^{2b -1}}{1+2k} \log(\frac{a^\alpha}{r}) + \frac{4 a^{\alpha 2b}}{\alpha r}\sum_{\substack{j=0 \\ j \neq k}}^\infty \frac{(\frac{r}{a^\alpha})^{2b} - (\frac{r}{a^\alpha})^{\frac{2j}{\alpha} + \frac{1}{\alpha}}}{(\frac{2j}{\alpha} + \frac{1}{\alpha})^2 - (2b)^2} \Big) dr
&  \text{if $2b= \frac{1+2k}{\alpha}$, $k \in \N_{\geq 1}$}.
\end{cases}
\end{align*}
If $2b<\frac{1}{\alpha}$, then, as $R\to 0$,
\begin{align*}
& 2\int_0^R \frac{4 a^{\alpha 2b}}{\alpha \pi r} \sum_{j=0}^\infty \frac{(\frac{r}{a^\alpha})^{2b} - (\frac{r}{a^\alpha})^{\frac{2j}{\alpha} + \frac{1}{\alpha}}}{(\frac{2j}{\alpha} + \frac{1}{\alpha})^2 - (2b)^2}  dr = 2\int_0^R \frac{4 a^{\alpha 2b}}{\alpha \pi r} \sum_{j=0}^\infty \frac{(\frac{r}{a^\alpha})^{2b}}{(\frac{2j+1}{\alpha})^2 - (2b)^2}  dr + \bigO(R^{\frac{1}{\alpha}}) \\
& = \frac{8}{\alpha \pi} \sum_{j=0}^\infty \frac{1}{(\frac{2j+1}{\alpha})^2 - (2b)^2}  \int_0^R r^{2b-1}dr + \bigO(R^{\frac{1}{\alpha}}) = \frac{\tan(\pi \alpha b)}{2b^{2}}R^{2b} + \bigO(R^{\frac{1}{\alpha}}),
\end{align*}
where for the last step we have used \cite[Equation 1.421.1]{GRtable}.
Hence, for any $\epsilon > 0$,
\begin{align*}
\nu_{b}\big(\partial S_{a^\alpha} \cap B_{R}(0)\big) 
\geq \begin{cases}
(1-\epsilon) \frac{\tan(\pi \alpha b)}{2b^2} R^{2b}
& \text{if $2b < \frac{1}{\alpha}$}, 
	\\
(1-\epsilon) \frac{2}{\pi b} R^{2b} \log(\frac{1}{R}),
&  \text{if $2b= \frac{1}{\alpha}$},
	\\
c R^{\frac{1}{\alpha}}
& \text{if $2b > \frac{1}{\alpha}$},
\end{cases}
\end{align*}
for all small enough $R > 0$.
Employing this estimate in (\ref{nupartialOmegaBr0}), we conclude that, for each $\epsilon > 0$,
\begin{align*}
\nu(\partial \Omega \cap B_r(0)) 
\geq \begin{cases}
(1-\epsilon) \frac{\tan(\pi \alpha b)}{2b^2} r^{2b}
& \text{if $2b < \frac{1}{\alpha}$}, 
	\\
(1-\epsilon) \frac{2}{\pi b} r^{2b} \log(\frac{1}{r}),
&  \text{if $2b= \frac{1}{\alpha}$},
	\\
c r^{\frac{1}{\alpha}}
& \text{if $2b > \frac{1}{\alpha}$},
\end{cases}
\end{align*}
for all small enough $r > 0$. This establishes also the desired lower bounds in (\ref{nubounds}). 

Finally, assume that $d\mu(z) \asymp |z-z_{0}|^{2b-2}d^{2}z$ as $z\to z_0$ so that (\ref{c1muc2}) holds for some $c_1, c_2, \rho_0$.
The estimate (\ref{nuestimate}) follows by applying (\ref{nubounds}) to $\nu_1 := \mathrm{Bal}(\mu_1,\partial \Omega)$ and $\nu_2 := \mathrm{Bal}(\mu_2,\partial \Omega)$, where the measures $\mu_1$ and $\mu_2$ are defined by
$$d\mu_j(z) = \begin{cases} 
 |z|^{2b-2}d^{2}z & \text{if $|z| < \rho_0$},
	\\
\frac{1}{c_j} d\mu(z) & \text{if $|z| \geq \rho_0$}, 
\end{cases} \quad j = 1,2,$$
and noting that
$$c_1 \nu_1(\partial \Omega \cap B_r(z_0)) \leq \nu(\partial \Omega \cap B_r(z_0)) \leq c_2 \nu_2(\partial \Omega \cap B_r(z_0)).$$
The proof of Theorem \ref{mainth} is complete.

\section{Proof of Theorem \ref{mainth2}}\label{proofsec2}
Suppose $\Omega$ is an open subset of $\C^*$ satisfying $(i)$ and $(ii)$ of Theorem \ref{mainth2} with $z_0 = 0$ and some $m \geq 1$. Let $\mu$ be a non-negative measure of finite total mass on $\Omega$ such that $d\mu(z) = (1 +o(1)) |z-z_{0}|^{2b-2}d^{2}z$ as $z\to z_0$. Let $\alpha := \max_{1 \leq j \leq m} \alpha_j$ and let $\nu := \mathrm{Bal}(\mu,\partial \Omega)$. Let $\nu_j := \mathrm{Bal}(\mu|_{U_j},\partial U_j)$ be the balayage of the restriction of $\mu$ to the component $U_j$ of $\Omega \cap B_{\rho_0}(z_0)$.

The next lemma shows that up to terms of order $O(r^{1/\alpha})$, $\nu(\partial \Omega \cap B_r(0))$ is given by the sum of the contributions $\nu_j(\partial U_j \cap B_r(0))$ from the $m$ corners. In other words, the contributions from the $m$ corners decouple and can be computed locally up to terms of order $O(r^{1/\alpha})$. 

\begin{lemma}[Decoupling and localization]\label{decouplinglemma}
There is a constant $C>0$ such that 
\begin{align}\label{decouplingestimate}
\nu(\partial \Omega \cap B_r(0)) - C r^{\frac{1}{\alpha}} \leq \sum_{j=1}^m \nu_j(\partial U_j \cap B_r(0)) \leq \nu(\partial \Omega \cap B_r(0))
\end{align}
for all sufficiently small $r > 0$.
\end{lemma}
\begin{proof}
By definition,
$$\nu(\partial \Omega \cap B_r(0)) = \int_\Omega \omega(z, \partial \Omega \cap B_r(0), \Omega) d\mu(z)$$
and, for $j = 1, \dots, m$,
$$\nu_j(\partial U_j \cap B_r(0)) = \int_{U_j} \omega(z, \partial U_j \cap B_r(0), U_j) d\mu(z).$$
Since $U_j$, $j = 1, \dots, m$, are the connected components of $U := \Omega \cap B_{\rho_0}(0) = \cup_{j=1}^m U_j$, we have
\begin{align*}
\sum_{j=1}^m \nu_j(\partial U_j \cap B_r(0)) = &\; \sum_{j=1}^m \int_{U_j} \omega(z, \partial U_j \cap B_r(0), U_j) d\mu(z)
	\\
=&\; \sum_{j=1}^m \int_{U_j} \omega(z, \partial U \cap B_r(0), U) d\mu(z) = \int_{U} \omega(z, \partial \Omega \cap B_r(0), U) d\mu(z).
\end{align*}
Hence, using twice that $U \subset \Omega$,
\begin{align}\nonumber
\sum_{j=1}^m \nu_j(\partial U_j \cap B_r(0))
& \leq \int_{U} \omega(z, \partial \Omega \cap B_r(0), \Omega) d\mu(z)
	\\ \label{sumjleq}
& \leq \int_{\Omega} \omega(z, \partial \Omega \cap B_r(0), \Omega) d\mu(z)
=  \nu(\partial \Omega \cap B_r(0))
\end{align}
for all $r \in (0,\rho_0)$, which is the second inequality in (\ref{decouplingestimate}).

Let $r \in (0,\rho_0)$. By Kakutani's theorem, (see e.g. \cite[Theorem F.6 and page 477]{GM2005})
$$\omega(z, \partial \Omega \cap B_r(0), \Omega)
= \mathbb{P}(W_\infty(z) \in \partial \Omega \cap B_r(0))$$
where $W_\infty(z) \in \partial \Omega \cap B_r(0)$ is the event that a Brownian motion starting at $z$ exits $\Omega$ at a point in $\partial \Omega \cap B_r(0)$. 

Suppose $z \in U_j$ for some $j = 1, \dots, m$ and let $E_z$ be the event that the Brownian motion $W_t(z)$ starting at $z$ hits the set $A := \Omega \cap \partial B_{\rho_0}(0)$. We split the event $W_\infty(z) \in \partial \Omega \cap B_r(0)$ into two depending on whether the set $A$ is hit or not:
$$\mathbb{P}(W_\infty(z) \in \partial \Omega \cap B_r(0))
= \mathbb{P}(\{W_\infty(z) \in \partial \Omega \cap B_r(0)\} \setminus E_z) + \mathbb{P}(\{W_\infty(z) \in \partial \Omega \cap B_r(0)\} \cap E_z).$$
If $W_t(z)$ does not hit the arc $A$, then the Brownian motion stays in $U_j$ for all times. So, using Kakutani's theorem again,
$$\mathbb{P}(\{W_\infty(z) \in \partial \Omega \cap B_r(0)\} \setminus E_z) =  \omega(z, \partial U_j \cap B_r(0), U_j).$$
Thus
\begin{align}\label{omegapartialUj}
\omega(z, \partial U_j \cap B_r(0), U_j) \geq \omega(z, \partial \Omega \cap B_r(0), \Omega) - \mathbb{P}(\{W_\infty(z) \in \partial \Omega \cap B_r(0)\} \cap E_z).
\end{align}

On the other hand, if the Brownian motion starting at $z \in U_j$ hits $A$, then in order to exit $\Omega$ in $\partial \Omega \cap B_r(0)$, it must make it from some point in $A$ to $\partial \Omega \cap B_r(0)$. The probability for a Brownian motion starting at a point $z_1 \in A$ to exit $\Omega$ in $\partial \Omega \cap B_r(0)$ is $\omega(z_1, \partial \Omega \cap B_r(0), \Omega)$. Thus,
\begin{align}\label{mathbbPWinfty}
\mathbb{P}(\{W_\infty(z) \in \partial \Omega \cap B_r(0)\} \cap E_z)
\leq \mathbb{P}(E_z) \sup_{z_1 \in A} \omega(z_1, \partial \Omega \cap B_r(0), \Omega).
\end{align}
Since
$$\omega(z_1, \partial \Omega \cap B_r(0), \Omega)
= \sum_{k=1}^m \omega(z_1, \partial U_k \cap B_r(0), \Omega),$$
the estimate (\ref{omegaleq8pirzmultiple2}) yields
$$\sup_{z_1 \in A} \omega(z_1, \partial \Omega \cap B_r(0), \Omega)
 \leq
 \sup_{z_1 \in A} \sum_{k=1}^m \omega(z_1, \partial U_k \cap B_r(0), \Omega)
 \leq \sum_{k=1}^m \frac{8}{\pi} \bigg(\frac{r}{\rho_0}\bigg)^{\frac{1}{\alpha_k}} 
\left(1 +  C_k \rho_0^{\gamma_k }\right)^{\frac{1}{\alpha_k  \gamma_k}}.$$
Using this estimate and the fact that $\mathbb{P}(E_z) \leq 1$ in (\ref{mathbbPWinfty}), and then substituting the resulting inequality into (\ref{omegapartialUj}), we arrive at
$$\omega(z, \partial U_j \cap B_r(0), U_j) \geq \omega(z, \partial \Omega \cap B_r(0), \Omega) - \sum_{k=1}^m \frac{8}{\pi} \bigg(\frac{r}{\rho_0}\bigg)^{\frac{1}{\alpha_k}} 
\left(1 +  C_k \rho_0^{\gamma_k }\right)^{\frac{1}{\alpha_k  \gamma_k}}$$
for each $z \in U_j$.
Integrating with respect to $d\mu(z)$ over $U_j$, we obtain
$$\nu_j(\partial U_j \cap B_r(0)) \geq \int_{U_j} \omega(z, \partial \Omega \cap B_r(0), \Omega) d\mu(z) - \sum_{k=1}^m \frac{8}{\pi} \bigg(\frac{r}{\rho_0}\bigg)^{\frac{1}{\alpha_k}} 
\left(1 +  C_k \rho_0^{\gamma_k }\right)^{\frac{1}{\alpha_k  \gamma_k}} \mu(U_j).$$
Summing over $j$ from $1$ to $m$ and then using that $\cup_{j=1}^m U_j = \Omega \setminus (\Omega \setminus B_{\rho_0}(0)) =\Omega \cap B_{\rho_0}(0)$, we find
$$\sum_{j=1}^m \nu_j(\partial U_j \cap B_r(0)) 
\geq \int_{\cup_{j=1}^m U_j} \omega(z, \partial \Omega \cap B_r(0), \Omega) d\mu(z) - \sum_{k=1}^m \frac{8}{\pi} \bigg(\frac{r}{\rho_0}\bigg)^{\frac{1}{\alpha_k}} 
\left(1 +  C_k \rho_0^{\gamma_k }\right)^{\frac{1}{\alpha_k  \gamma_k}} \sum_{j=1}^m \mu(U_j).$$
$$
= \nu(\partial \Omega \cap B_r(0)) - \int_{\Omega \setminus B_{\rho_0}(0)} \omega(z, \partial \Omega \cap B_r(0), \Omega) d\mu(z)
 - \sum_{k=1}^m \frac{8}{\pi} \bigg(\frac{r}{\rho_0}\bigg)^{\frac{1}{\alpha_k}} 
\left(1 +  C_k \rho_0^{\gamma_k }\right)^{\frac{1}{\alpha_k  \gamma_k}} \mu(\Omega \cap B_{\rho_0}(0)).$$
The integral on the right-hand side can be estimated using (\ref{omegaleq8pirzmultiple2}):
\begin{align*}
\int_{\Omega \setminus B_{\rho_0}(0)} \omega(z, \partial \Omega \cap B_r(0), \Omega) d\mu(z)
& = \sum_{k=1}^m \int_{\Omega \setminus B_{\rho_0}(0)} \omega(z, \partial U_k \cap B_r(0), \Omega) d\mu(z)
	\\
& \leq \sum_{k=1}^m 
\frac{8}{\pi} \bigg(\frac{r}{\rho_0}\bigg)^{\frac{1}{\alpha_k}} 
\left(1 +  C_k \rho_0^{\gamma_k }\right)^{\frac{1}{\alpha_k  \gamma_k}}
 \mu(\Omega \setminus B_{\rho_0}(0))
 \end{align*}
for all $r \in (0,\rho_0)$. It transpires that
\begin{align}\label{sumjgeq}
\sum_{j=1}^m \nu_j(\partial U_j \cap B_r(0)) 
\geq \nu(\partial \Omega \cap B_r(0)) - \sum_{k=1}^m 
\frac{8}{\pi} \bigg(\frac{r}{\rho_0}\bigg)^{\frac{1}{\alpha_k}} 
\left(1 +  C_k \rho_0^{\gamma_k }\right)^{\frac{1}{\alpha_k  \gamma_k}} \mu(\Omega).
\end{align}
Since $\alpha = \max_{1 \leq j \leq m} \alpha_j$, the first inequality in (\ref{decouplingestimate}) follows.
\end{proof}

For each $j$, we can estimate $\nu_j(\partial U_j \cap B_r(0))$ by applying Theorem \ref{mainth} to the domain $U_j$. Summing over $j$ from $1$ to $m$, this yields, for all sufficiently small $r > 0$,
\begin{align*}\nonumber
& (1-\epsilon) \sum_{j=1}^m \frac{\tan(\pi \alpha_j b)}{2b^2} r^{2b} \leq \sum_{j=1}^m \nu_j(\partial U_j \cap B_r(0))
  \leq (1+\epsilon) \sum_{j=1}^m \frac{\pi  \alpha_j}{2b} \bigg(1  + \frac{16 b}{\pi (\frac{1}{\alpha_j} - 2b)}\bigg)  r^{2b}, \hspace{-.15cm}&&  \text{$2b < \frac{1}{\alpha}$}, 
	\\ \nonumber 
& m_{\alpha} (1-\epsilon) \frac{2}{\pi b} r^{2b} \log(\tfrac{1}{r}) - c r^{2b} 
 \leq
\sum_{j=1}^m \nu_j(\partial U_j \cap B_r(0))
\leq m_{\alpha} (1+\epsilon) \frac{4}{b} r^{2b} \log(\tfrac{1}{r}) + C r^{2b}, &&  \text{$2b= \frac{1}{\alpha}$}, 	
	\\ 
& c r^{\frac{1}{\alpha}} \leq
\sum_{j=1}^m \nu_j(\partial U_j \cap B_r(0))
\leq C r^{\frac{1}{\alpha}}, & &  \text{$2b > \frac{1}{\alpha}$},
\end{align*}
where $m_{\alpha}$ is the number of $\alpha_j$ such that $\alpha_j = \alpha$. Combining these inequalities with Lemma \ref{decouplinglemma}, the estimates in (\ref{numultiplebounds}) follow. The fact that (\ref{nuestimate}) holds if $d\mu(z) \asymp |z-z_{0}|^{2b-2}d^{2}z$ as $z\to z_0$ then follows in the same way as in the proof of Theorem \ref{mainth}.

\section{Application}\label{section:application}

In this section, we highlight the relevance of Theorems \ref{mainth} and \ref{mainth2} in the study of two-dimensional Coulomb gases. The planar Coulomb gas model for $n$ points with external potential $Q:\C\to \R\cup \{+\infty\}$ is the probability measure
\begin{align}
& \frac{1}{Z_{n}}\prod_{1\leq j<k \leq n}|\mathrm{z}_{j}-\mathrm{z}_{k}|^{\beta} \prod_{j=1}^{n} e^{-n \frac{\beta}{2} Q(\mathrm{z}_{j})}d^{2}\mathrm{z}_{j}, & & \mathrm{z}_{1},\ldots,\mathrm{z}_{n}\in \C, \label{general density intro}
\end{align}
where $Z_{n}$ is the normalization constant,  and $\beta >0$ is the inverse temperature. For $\beta=2$, \eqref{general density intro} is also the law of the complex eigenvalues of a class of $n \times n$ random normal matrices (see e.g. \cite{BFreview}). Standard equilibrium convergence theorems imply, under quite general assumptions on $Q$, that the points $\mathrm{z}_{1},\ldots,\mathrm{z}_{n}$ will condensate (as $n \to \infty$ and with high probability) on the support $S$ of an equilibrium measure $\mu$. The measure $\mu$ is defined as the unique measure minimizing 
\begin{align*}
\sigma \mapsto \int \log \frac{1}{|z-w|} d\sigma(z)d\sigma(w) + \int Q(z) d\sigma(z)
\end{align*}
among all Borel probability measures $\sigma$ on $\C$ \cite{SaTo}. For example, if $Q(z) = |z|^{2b}$, then 
\begin{align*}
d\mu(z) = \frac{b^{2}}{\pi}|z|^{2b-2}\chi_{S}(z)d^{2}z, \qquad S=\{z \in \C: |z| \leq b^{-\frac{1}{2b}}\},
\end{align*}
where $\chi_{S}$ is the indicator function of $S$, see also Figure \ref{fig:circular sector} (row 1). 
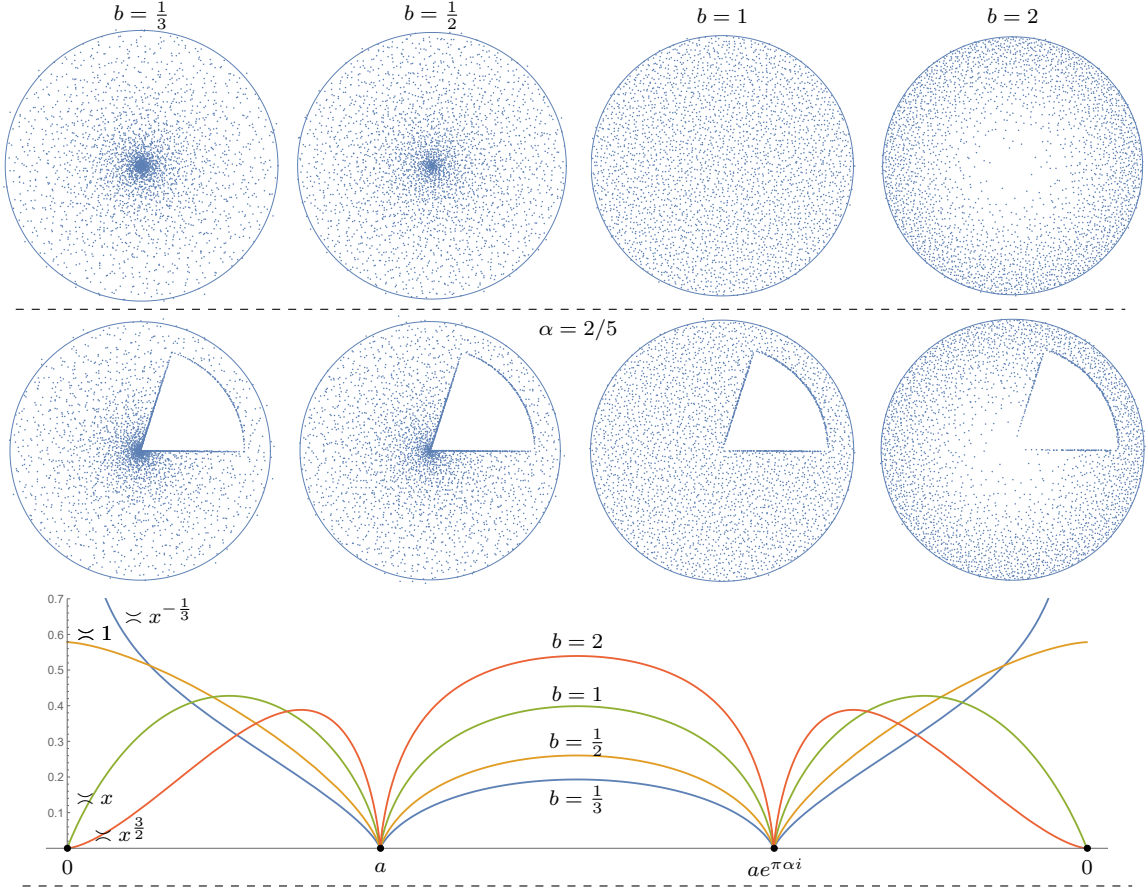
\begin{figure}
\begin{tikzpicture}[master]
\node at (0,0) {\includegraphics[width=3.63cm]{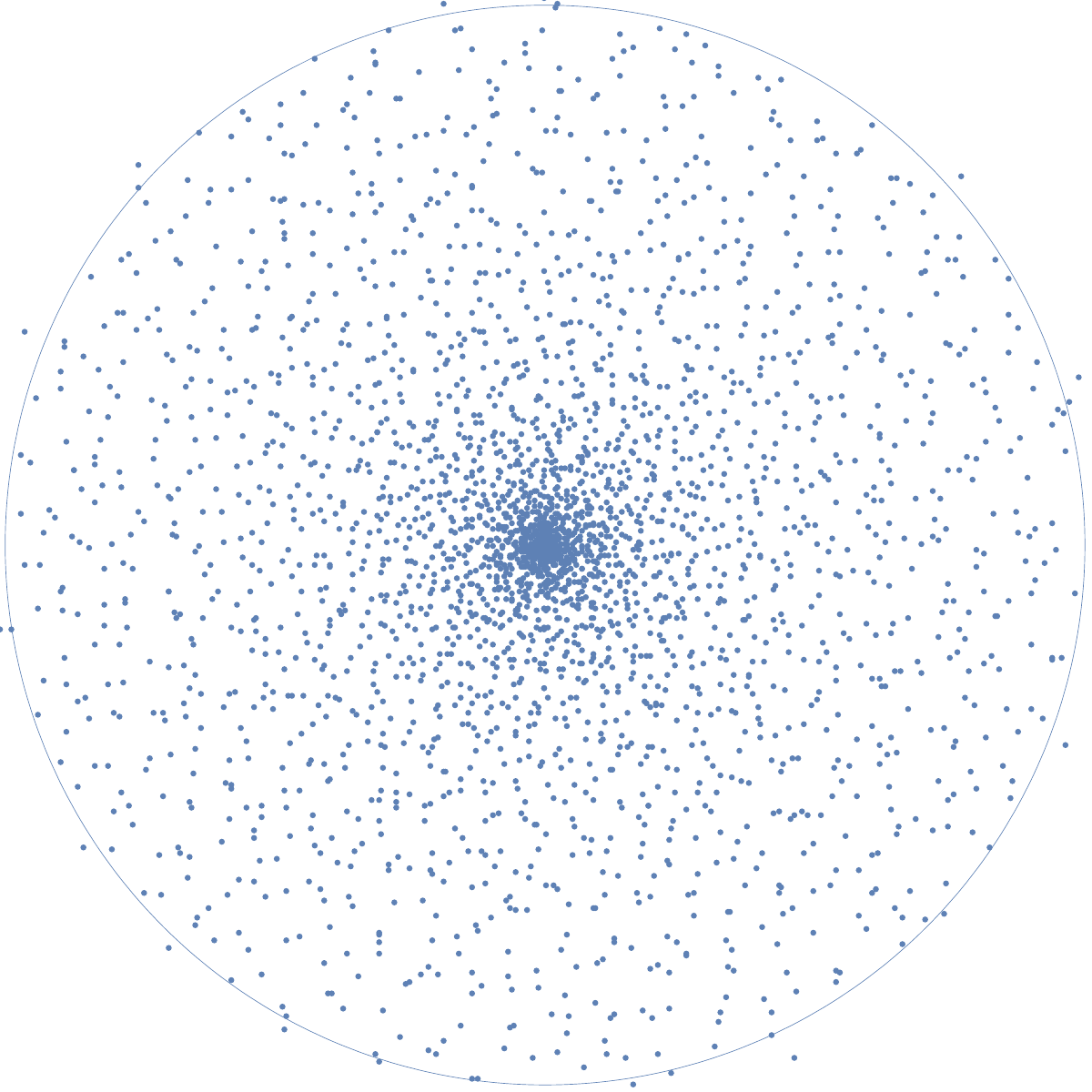}};
\node at (0,2) {\footnotesize $b=\frac{1}{3}$};
\end{tikzpicture} \hspace{-0.21cm}
\begin{tikzpicture}[slave]
\node at (0,0) {\includegraphics[width=3.63cm]{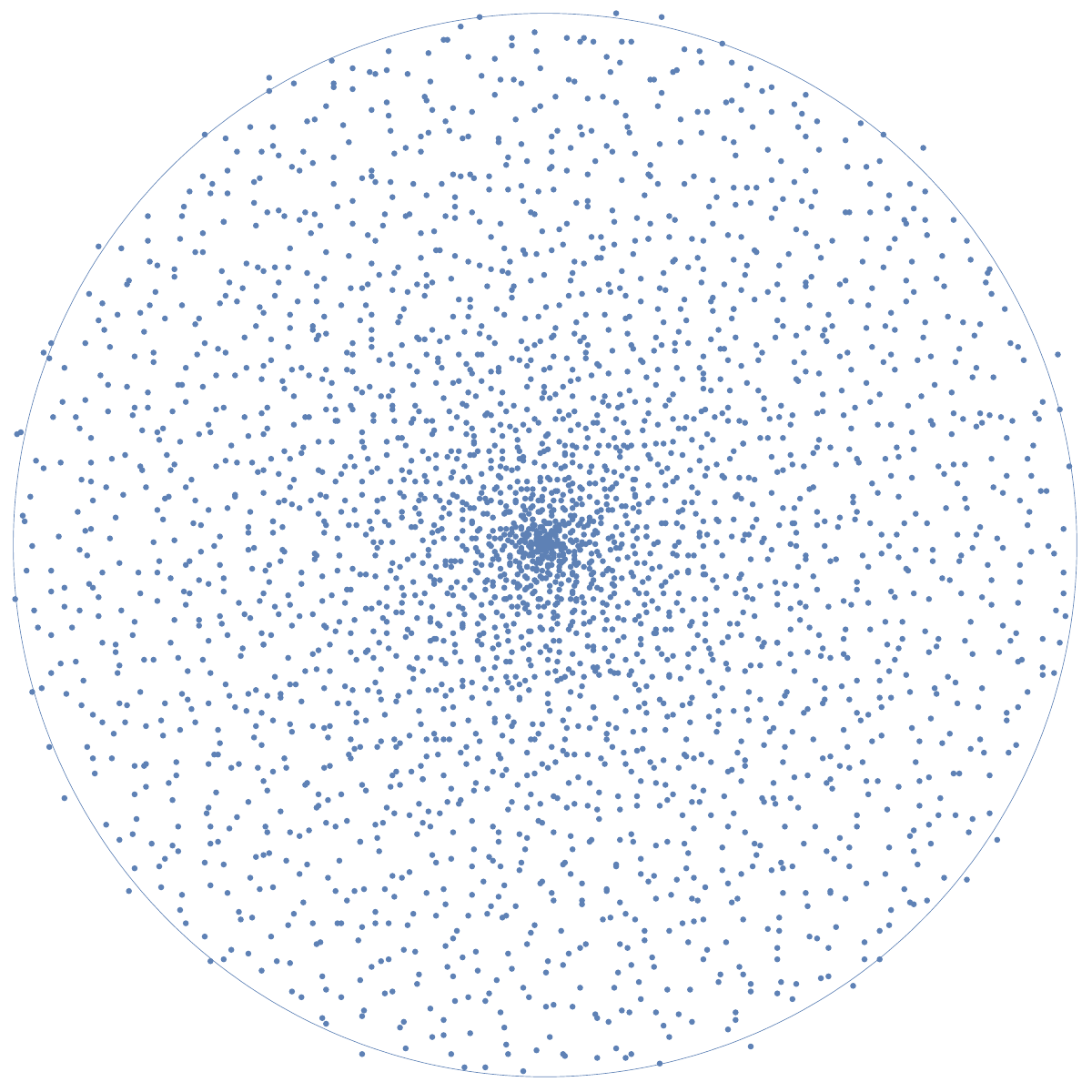}};
\node at (0,2) {\footnotesize $b=\frac{1}{2}$};
\end{tikzpicture} \hspace{-0.21cm}
\begin{tikzpicture}[slave]
\node at (0,0) {\includegraphics[width=3.63cm]{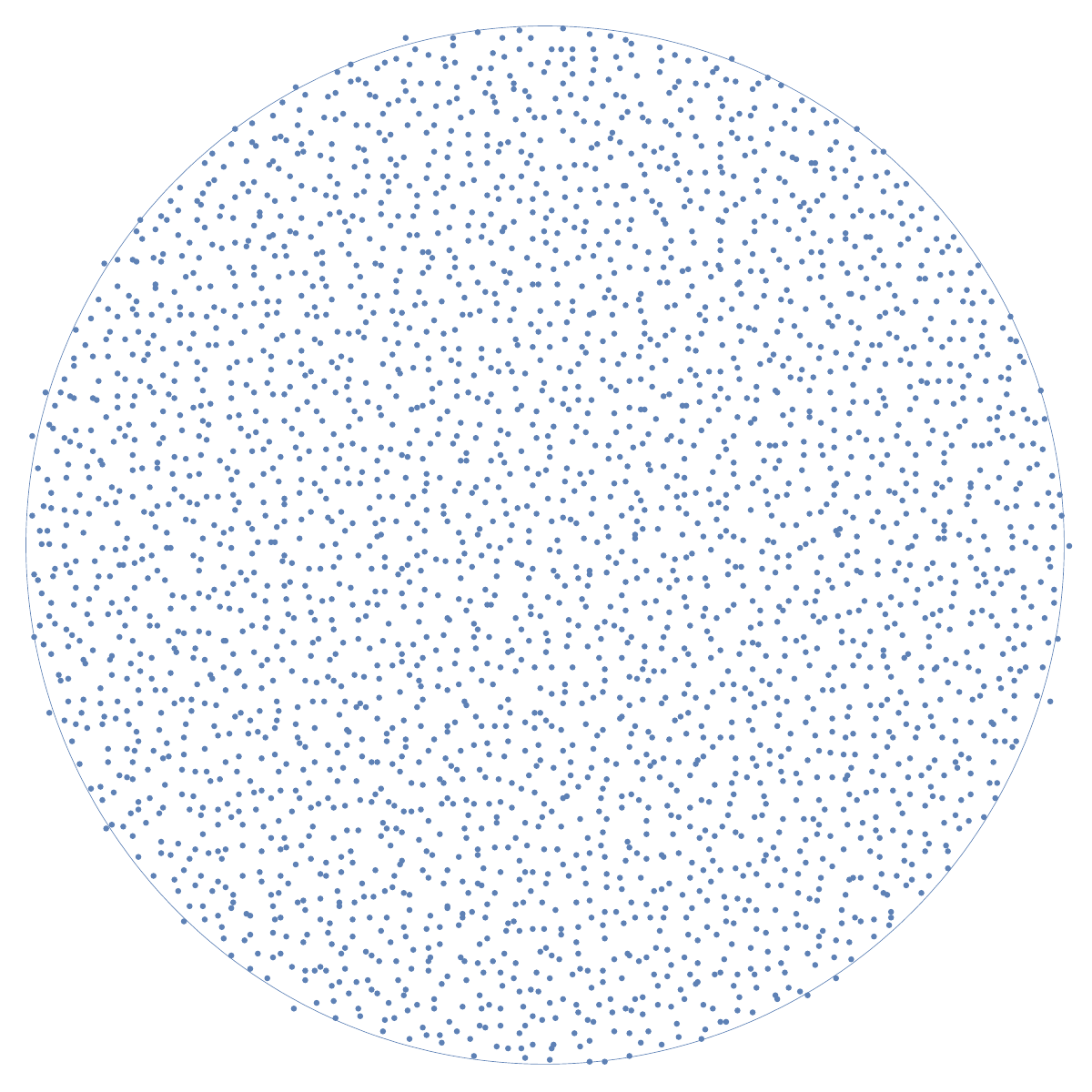}};
\node at (0,2) {\footnotesize $b=1$};
\draw[dashed] (-9.3,-1.9)--(5.3,-1.9);
\end{tikzpicture} \hspace{-0.21cm}
\begin{tikzpicture}[slave]
\node at (0,0) {\includegraphics[width=3.63cm]{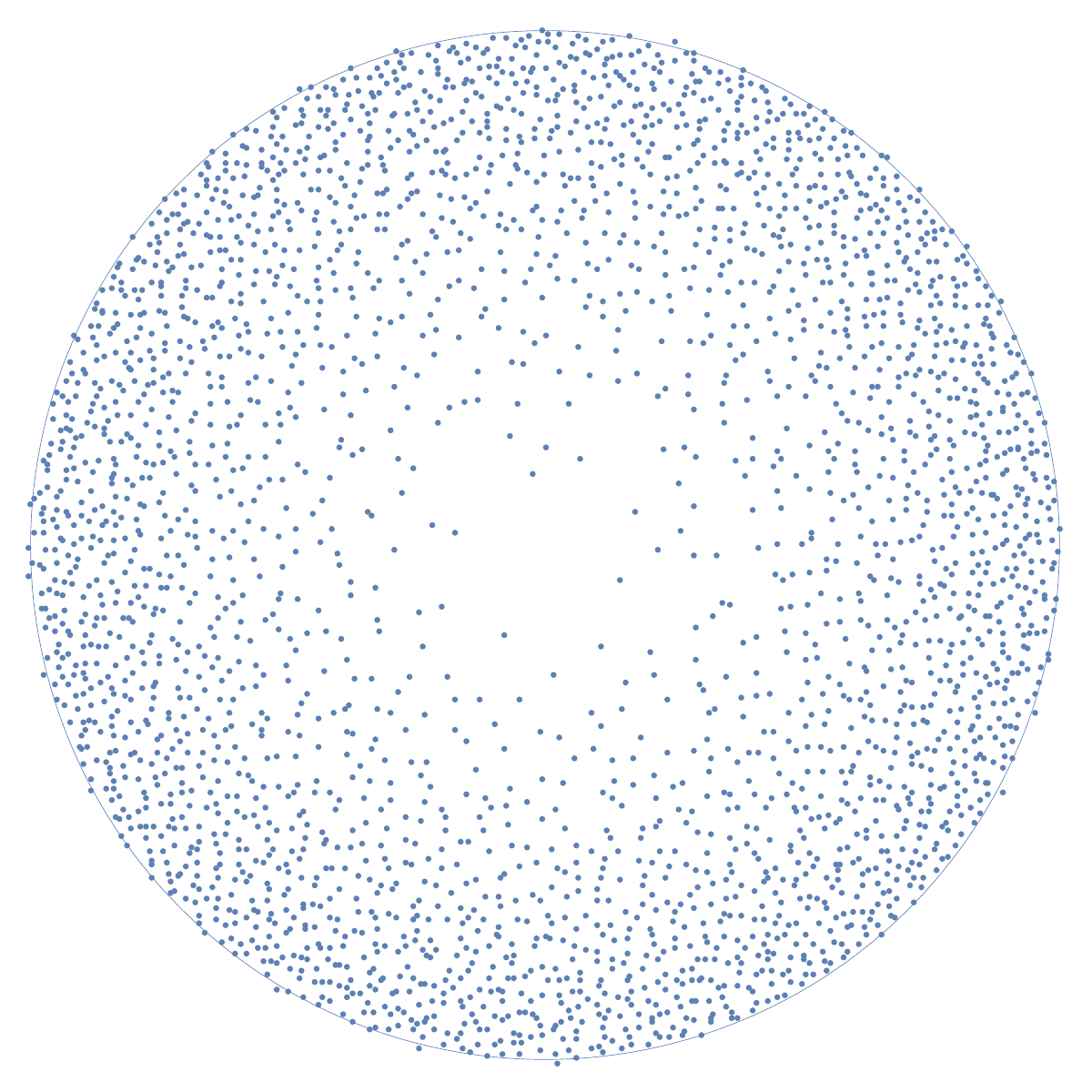}};
\node at (0,2) {\footnotesize $b=2$};
\end{tikzpicture} \\[-0.5cm]
\begin{tikzpicture}[slave]
\node at (0,0) {\includegraphics[width=3.63cm]{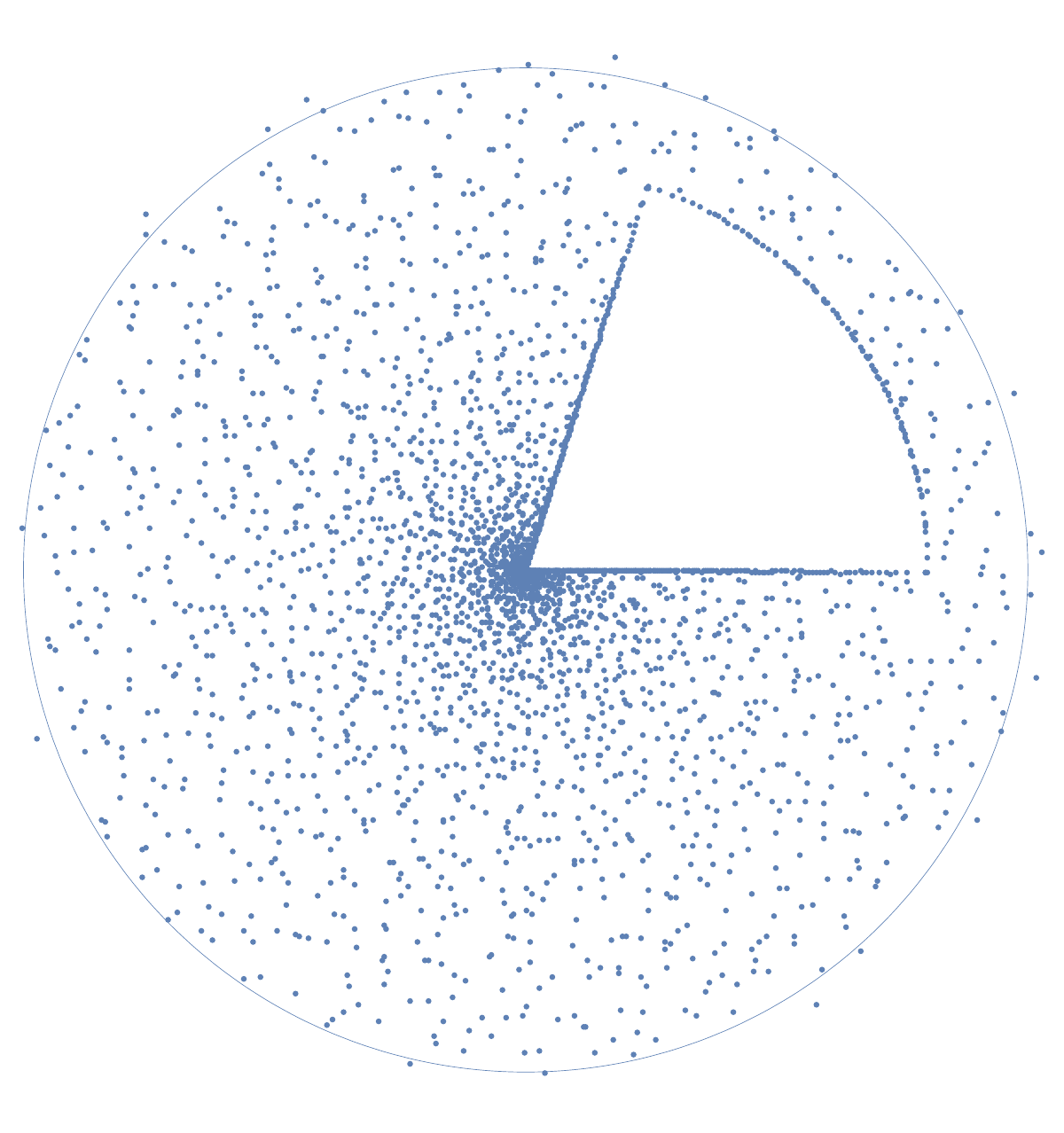}};
\end{tikzpicture} \hspace{-0.21cm}
\begin{tikzpicture}[slave]
\node at (0,0) {\includegraphics[width=3.63cm]{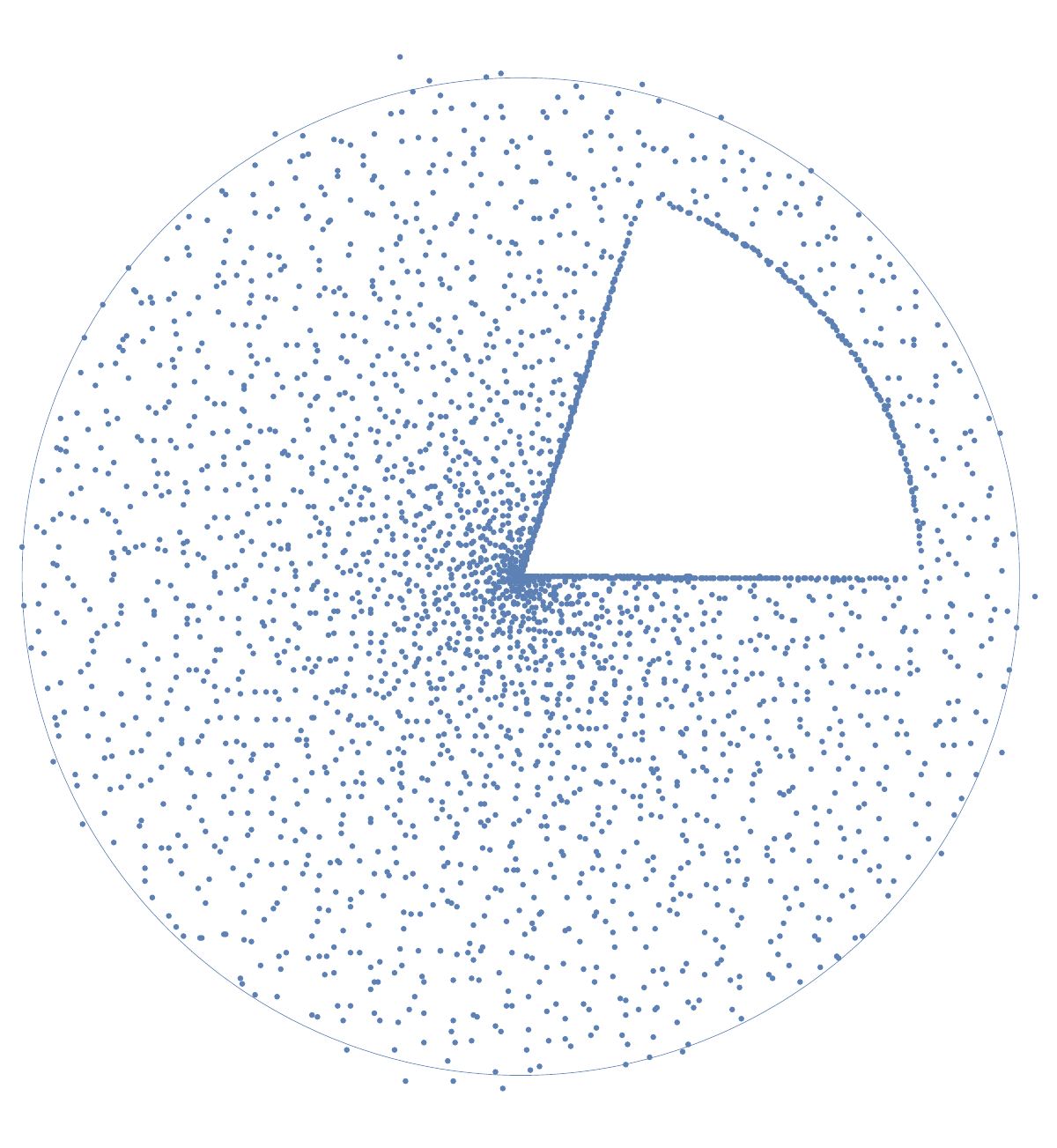}};
\end{tikzpicture} \hspace{-0.21cm}
\begin{tikzpicture}[slave]
\node at (0,0) {\includegraphics[width=3.63cm]{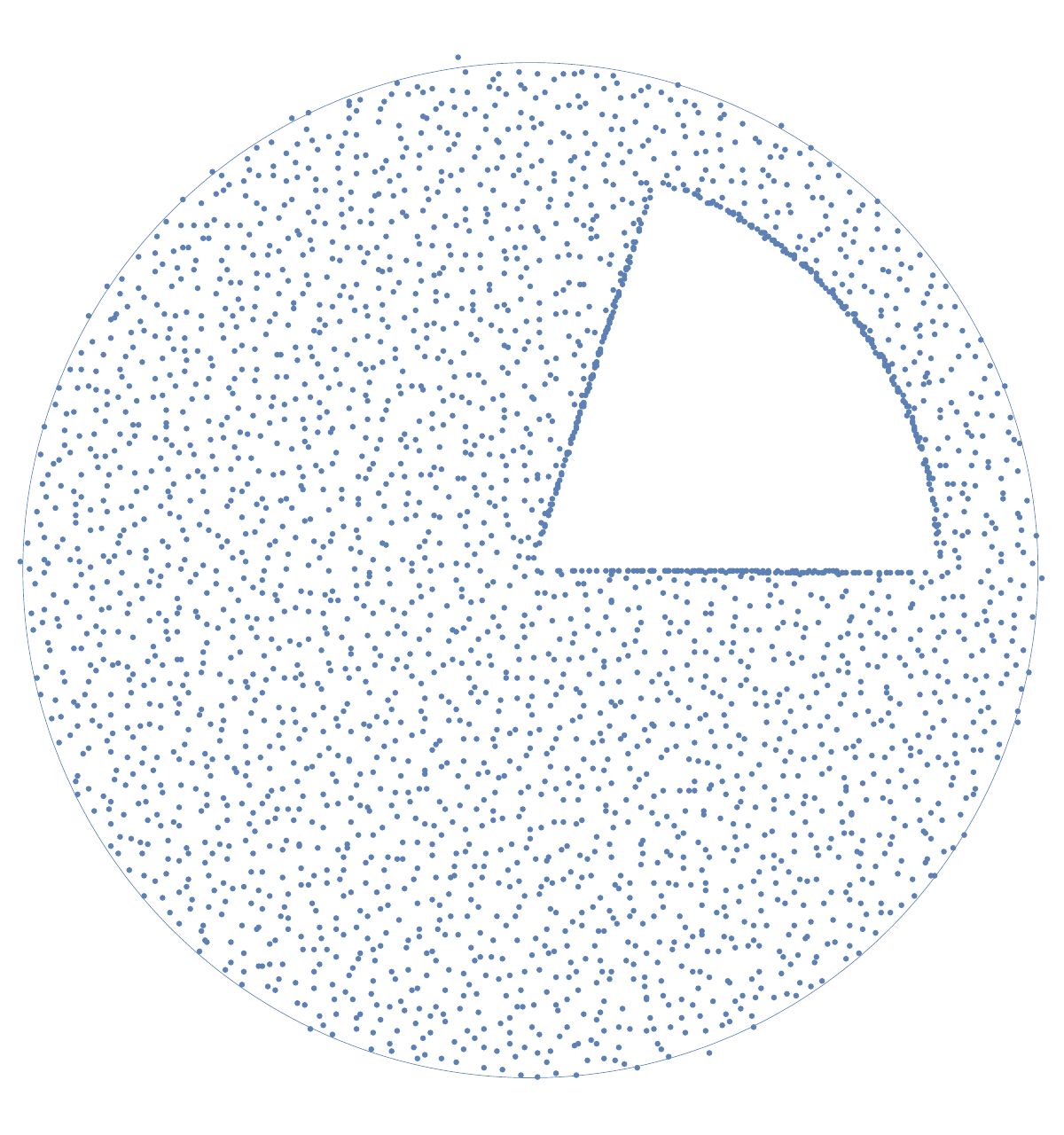}};
\node at (-1.9,1.6) {\footnotesize $\alpha=2/5$};
\end{tikzpicture} \hspace{-0.21cm}
\begin{tikzpicture}[slave]
\node at (0,0) {\includegraphics[width=3.63cm]{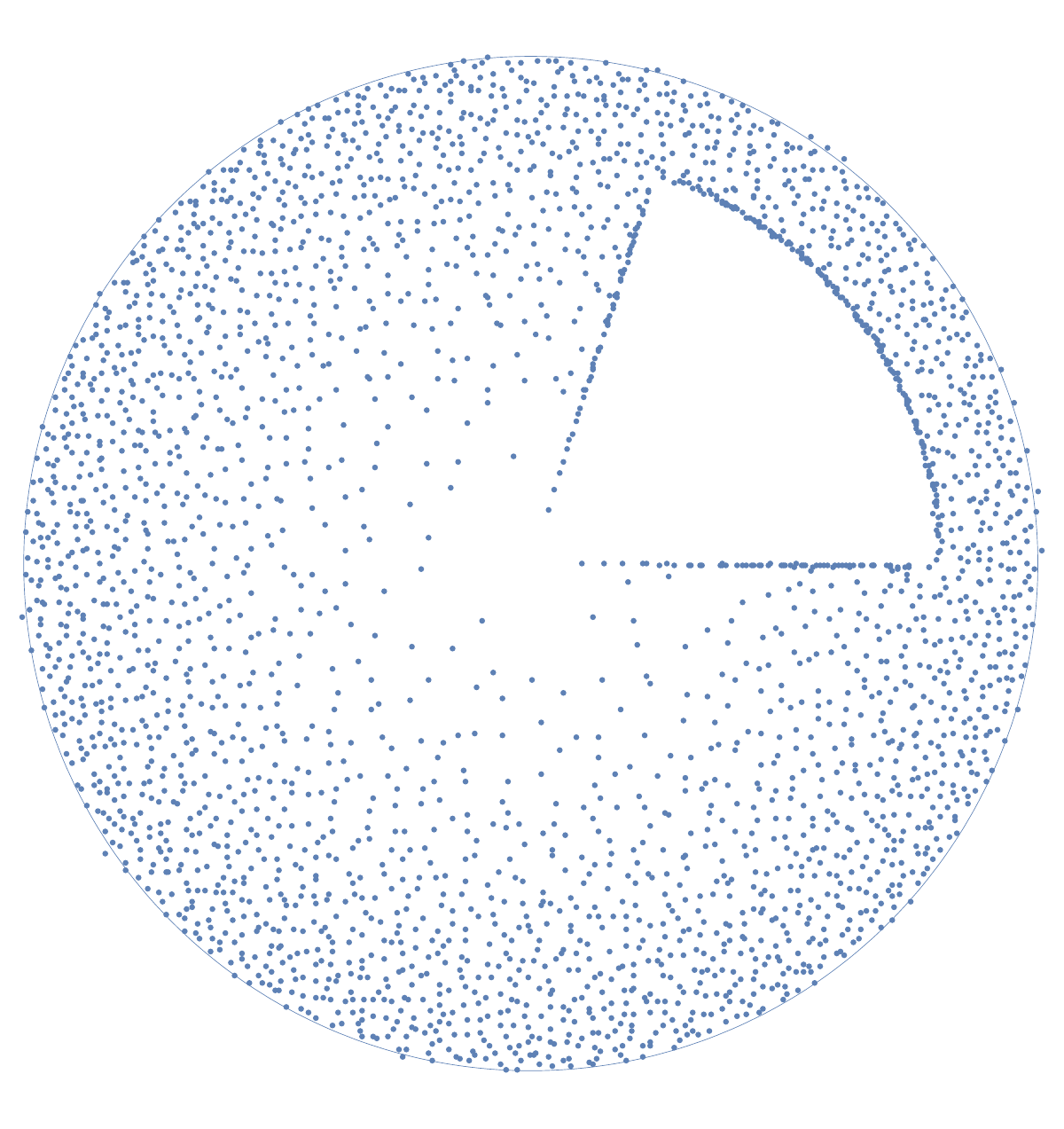}};
\end{tikzpicture} \\[-0.9cm]
\begin{center}
\begin{tikzpicture}[master]
\node at (0,0) {\includegraphics[width=14cm]{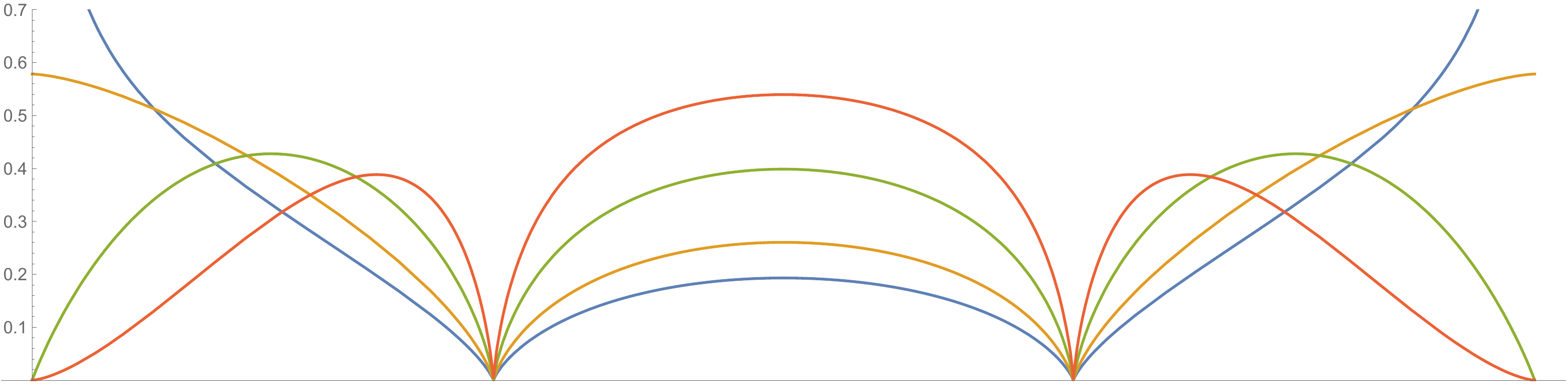}};
\node at (0,-1) {\footnotesize $b=\frac{1}{3}$};
\node at (0,-0.25) {\footnotesize $b=\frac{1}{2}$};
\node at (0,0.4) {\footnotesize $b=1$};
\node at (0,1.1) {\footnotesize $b=2$};
\node at (-5.5,1.5) {\footnotesize $\asymp x^{-\frac{1}{3}}$};
\node at (-6.35,1.2) {\footnotesize $\asymp 1$};
\node at (-6.35,-1) {\footnotesize $\asymp x$};
\node at (-6,-1.4) {\footnotesize $\asymp x^{\frac{3}{2}}$};
\node at (-6.35,1.2) {\footnotesize $\asymp 1$};
\draw[fill] (-6.71,-1.65) circle (0.04);
\node at (-6.71,-1.9) {\footnotesize $0$};
\draw[fill] (-2.59,-1.65) circle (0.04);
\node at (-2.59,-1.9) {\footnotesize $a$};
\draw[fill] (2.59,-1.65) circle (0.04);
\node at (2.61,-1.9) {\footnotesize $a e^{ \pi \alpha i}$};
\draw[fill] (6.71,-1.65) circle (0.04);
\node at (6.71,-1.9) {\footnotesize $0$};
\draw[dashed] (-7.3,-2.15)--(7.3,-2.15);
\end{tikzpicture}
\end{center}
\vspace{-0.5cm}\caption{\label{fig:circular sector} (Taken from \cite{C2023}) Row 1: the point process \eqref{general density intro} with $Q(z) = |z|^{2b}$ and the indicated values of $b$. In each plot, the thin blue circle is $\partial S = \{z: |z|=b^{-\frac{1}{2b}}\}$. Row 2: the point process \eqref{general density intro} with $Q$ as in \eqref{Q example hard}, $Q_{0}(z) = |z|^{2b}$ and $\Omega=\{re^{i\theta} : 0<r<a, \, 0<\theta < \pi\alpha\}$, $\alpha=2/5$, and $a=0.8b^{-\frac{1}{2b}}$. Row 3: the normalized density $\frac{d\nu(z)/|dz|}{\nu(\partial \Omega)}$ for $\alpha=2/5$ and $a=0.8b^{-\frac{1}{2b}}$.}
\end{figure}
More generally, if $Q$ is smooth on $S$, then by \cite[Theorem II.1.3]{SaTo} $\mu$ is absolutely continuous with respect to $d^{2}z$ and given by $\frac{1}{4\pi}\Delta Q(z) \chi_{S}(z)d^{2}z$. The situation is more complicated if $Q=+\infty$ on a certain subset $\Omega \subset \C$: this so-called ``hard wall constraint" confines the points to lie in $\C\setminus \Omega$. Recent studies on Coulomb gases with hard edges include \cite{A2018, AR2017, Berezin, BP2024, C2023}. Suppose for example that $Q_{0}$ is smooth on $\C$, let $S_{0}$ be the support of the associated equilibrium measure $\mu_{0}$, let $\Omega \subset \C^{*}$ be a finitely connected Jordan domain such that $\partial \Omega \subset S_{0}$, and define
\begin{align}\label{Q example hard}
Q(z) := \begin{cases}
Q_{0}(z), & \mbox{if } z\notin \Omega, \\
+\infty, & \mbox{if } z \in \Omega.
\end{cases}
\end{align}
It is proved in \cite{A2018, AR2017, C2023} that the associated equilibrium measure is $\mu:= \mu_{0}|_{S_{0}\setminus \Omega} + \nu$, where $\nu:=\mathrm{Bal}(\mu_{0}|_{\Omega},\partial \Omega)$. The point process \eqref{general density intro} with $Q$ as in \eqref{Q example hard}, $Q_{0}(z) := |z|^{2b}$, and $\Omega:=\{re^{i\theta} : 0<r<a, \, 0<\theta < \pi \alpha \}$, $\alpha=2/5$, $a=0.8b^{-\frac{1}{2b}}$ is illustrated in Figure \ref{fig:circular sector} (row 2) for several values of $b$, and the balayage measure $\nu$ is illustrated in Figure \ref{fig:circular sector} (row 3).

The universality conjecture asserts that as $n \to \infty$ the limiting local statistical properties of the random points $\mathrm{z}_{1},\ldots,\mathrm{z}_{n}$ around a given $z_{0}\in \C$ depend only on $\beta$ and on the behavior of $\mu$ at $z_{0}$. Consider the hard wall case \eqref{Q example hard}, and suppose that $d\mu_{0}(z) \asymp |z-z_{0}|^{2b-2}d^{2}z$ as $z\to z_0$ for some $b > 0$, and that $\Omega$ has a H\"older-$C^1$ corner of opening $\pi \alpha$ at $z_{0}$ for some $\alpha \in (0,2]$. Theorem \ref{mainth} implies that 
\begin{align}\label{rate}
d\mu(z)/|dz| \asymp \begin{cases}
|z-z_{0}|^{\min\{2b,\frac{1}{\alpha}\}-1}, & \mbox{if } 2b \neq \frac{1}{\alpha}, \\
|z-z_{0}|^{2b-1} \log \frac{1}{|z-z_{0}|}, & \mbox{if } 2b = \frac{1}{\alpha},
\end{cases} \qquad \mbox{as } z\to z_{0}, \; z\in \partial \Omega,
\end{align}
where $|dz|$ is the arclength measure on $\partial U$. Note that the rate of convergence (or blow-up) in \eqref{rate} depends on $\Omega$ only through $\alpha$. If $\Omega$ has multiple corners at $z_{0}$, then similar estimates can be obtained from Theorem \ref{mainth2}.
For $\beta=2$, universality results on local statistics near hard edges can be found in \cite{Seo, ACC2023} in the case where $\alpha=1$ and $b=1$. In view of the above, new universality classes are expected for other values of $\alpha$ and $b$.

\appendix 

\section{A technical lemma}
Lemma \ref{exerciselemma} stated below is used in Section \ref{proofsec} to obtain (\ref{fnear0}). The lemma is a special case of \cite[Exercise 3.4.1]{P1992}. For the reader's convenience, we include a proof. 

Let $\Omega_1$ be a connected Jordan domain such that $\partial \Omega_1$ has a H\"older-$C^1$ corner of opening $\pi \alpha$ for some $0 < \alpha \leq 2$ at $0\in \partial \Omega_1$, and such that one of the boundary arcs forming the corner is tangent to $(0,+\infty)$ at $0$. Hence, there exist $\gamma\in (0,1]$ and curves $C_{\pm} \subset \partial \Omega_1$  of class $C^{1,\gamma}$ ending at $0$ and lying on different sides of $0$ such that \eqref{argz0pialpha} holds. 

\begin{lemma}\label{exerciselemma}
Suppose $f$ is a conformal map from the open upper half-plane $\mathbb{H}$ onto $\Omega_1$ such that $f(0) = 0$, $C_+ \subset f(\R_+)$, and $C_- \subset f(\R_-)$. Suppose also that $\gamma < \frac{1}{2}$. Then
\begin{align}\label{lol2}
f(z) = c z^\alpha(1 + O(z^{\alpha \gamma})), \qquad \mbox{as } z \to 0, \; z \in  \bar{\mathbb{H}},
\end{align}
for some $c>0$.
\end{lemma}
\begin{proof}
Let $\Phi(w) = w^{1/\alpha}$, where the principal branch is used for the root, so that $\Phi$ maps the sector $\{z:\arg z \in (0,\pi \alpha)\}$ conformally onto the upper-half plane $\mathbb{H}$. Let $w_{\pm}$ be the parametrizations of $C_{\pm}$ given in Lemma \ref{nearcornerlemma}. By \eqref{wpm der near 0}, 
\begin{align*}
& (\Phi \circ w_{+})'(r) = \Phi'(w_{+}(r))w_{+}'(r) = \frac{1}{\alpha}r^{\frac{1}{\alpha}-1} \big(1 + O(r^{\gamma})\big), & & \mbox{as } r \to 0, \\
& (\Phi \circ w_{-})'(r) = \Phi'(w_{-}(r))w_{-}'(r) = -\frac{1}{\alpha}r^{\frac{1}{\alpha}-1} \big(1 + O(r^{\gamma})\big), & & \mbox{as } r \to 0.
\end{align*}
Integrating the above, and then replacing $r$ by $r^{\alpha}$, we find
\begin{align*}
& (\Phi \circ w_{+})(r^{\alpha}) = r\big(1 + O(r^{\alpha\gamma})\big), & & (\Phi \circ w_{-})(r^{\alpha}) = - r\big(1 + O(r^{\alpha\gamma})\big), & & \mbox{as } r \to 0,
\end{align*}
Using also that
\begin{align*}
& \frac{d}{dr}(\Phi \circ w_{+})(r^{\alpha}) 
= 1 + O(r^{\alpha\gamma}), & &  \frac{d}{dr}(\Phi \circ w_{-})(r^{\alpha}) 
= -1 + O(r^{\alpha \gamma}), & & \mbox{as } r \to 0,
\end{align*}
it follows that
$$r \mapsto \begin{cases}
(\Phi \circ w_{+})(r^{\alpha}), & r \geq 0,	\\
(\Phi \circ w_{-})((-r)^{\alpha}), & r < 0,
\end{cases}$$
is a Jordan curve of class $C^{1,\gamma^{\star}}$ parametrizing $\partial \Phi (\Omega_{1})$ in a small neighborhood of $0$, where $\gamma^{\star} = \min\{\alpha\gamma,1\}$.
The assumption $\gamma < \frac{1}{2}$, together with $\alpha\leq 2$, implies that $\gamma^{\star}=\alpha\gamma<1$. By the Kellogg–Warschawski theorem \cite[Theorem 3.6]{P1992}, the function $g = \Phi\circ f: \mathbb{H} \to \Phi(\Omega_{1})$ satisfies
\begin{align}\label{lol1}
g(z) = c^{1/\alpha}z\big(1+ O(z^{\alpha \gamma})\big), \qquad \mbox{as } z\to 0, \; z\in \bar{\mathbb{H}},
\end{align}
for some constant $c\in \C\setminus\{0\}$. Hence $f(z)=g(z)^{\alpha}$ satisfies \eqref{lol2}. Moreover, since $C_{+}\subset f(\R_{+})$ and $w_{+}'(0)=1$, we have $c>0$, which finishes the proof.
\end{proof}
\subsection*{Acknowledgements}
CC acknowledges support from the Swedish Research Council, Grant No. 2021-04626, and JL acknowledges support from the Swedish Research Council, Grant No. 2021-03877.

\end{document}